\documentclass[a4,12pt]{article}
\usepackage{graphicx}
\usepackage{subfig}
\usepackage{mathptmx}       % selects Times Roman as basic font
\usepackage{helvet}         % selects Helvetica as sans-serif font
\usepackage{courier}        % selects Courier as typewriter font
\usepackage{type1cm}        % activate if the above 3 fonts are
\usepackage{makeidx}         % allows index generation
\usepackage{graphicx}        % standard LaTeX graphics tool
\usepackage{multicol}        % used for the two-column index
\usepackage[bottom]{footmisc}% places footnotes at page bottom

\usepackage{subfig}
\usepackage{amsmath}
\usepackage{latexsym}

\author{Magnus Fontes}

\usepackage{amsmath}

\newtheorem{definition}{Definition}

\newtheorem{theorem}{Theorem}
\newtheorem{lemma}{Lemma}
\newtheorem{proposition}{Proposition}
\numberwithin{theorem}{section}
\numberwithin{equation}{section}
\numberwithin{definition}{section}
\numberwithin{corollary}{section}
\numberwithin{lemma}{section}
\makeindex

\begin{document}
\title{Statistical and knowledge supported visualization of multivariate data.}
\author{Magnus Fontes \\
Centre for Mathematical Sciences \\Lund University, Box 118, SE-22100, Lund, Sweden.\\
email: fontes@maths.lth.se} 
\date{}
\maketitle

{\bf Abstract}
In the present work we have selected a collection of statistical and mathematical tools 
useful for the exploration of multivariate data 
and we present them in a form that is meant to be 
particularly accessible to a classically trained mathematician.
We give self contained and streamlined introductions to
principal component analysis, multidimensional scaling and 
statistical hypothesis testing. 
Within the presented mathematical framework we then propose a 
general exploratory methodology for the investigation of 
real world high dimensional datasets that builds on 
statistical and knowledge supported visualizations.
We exemplify the proposed methodology by applying it to several different 
genomewide DNA-microarray datasets. 
The exploratory methodology should be seen as an embryo
that can be expanded and developed in many directions. 
As an example we point out some recent promising advances in the theory for 
random matrices that, if further developed, potentially could provide 
practically useful and theoretically well founded estimations of 
information content in dimension reducing visualizations. 
We hope that the present work can serve as an introduction to, and 
help to stimulate more
research within, the interesting and rapidly expanding 
field of data exploration.

\section{Introduction.}
In the scientific exploration of some real world phenomena a lack of detailed knowledge about 
governing first principles makes it hard to construct well-founded mathematical models for describing  and 
understanding observations.
In order to gain some preliminary understanding of involved mechanisms and to be able
to make some reasonable predictions we then often have to recur to purely statistical models.
Sometimes though, a stand alone and very general statistical approach 
fails to exploit the full exploratory 
potential for a given dataset. 
In particular a general statistical model a priori often does not incorporate 
all the accumulated field-specific expert knowledge that might exist
concerning a dataset under consideration.
In the present work we argue for the use of a set of statistical and knowledge supported visualizations 
as the backbone of the exploration of high dimensional multivariate 
datasets that are otherwise hard to model and analyze.
The exploratory methodology we propose is generic but we 
exemplify it by applying it to several different datasets 
coming from the field of molecular biology.
Our choice of example application field is in principle anyhow 
only meant to be
reflected in the list of references where we have consciously striven to give 
references that should be particularly useful and relevant for researchers interested in bioinformatics.
The generic case we have in mind is that we are given a set of 
observations of several different variables 
that presumably have some interrelations that we want to uncover. 
There exist many rich real world sources giving rise to interesting examples of such datasets
within the fields of e.g. finance, astronomy, meteorology or life science and the reader should
without difficulty be able to pick a favorite example to bear in mind. 
 
We will use separate but synchronized 
Principle Component Analysis (PCA) plots of both variables and samples to visualize datasets.
The use of separate but synchronized PCA-biplots that we argue 
for is not standard and we claim that it is particularly advantageous, compared to 
using traditional PCA-biplots, when the datasets under 
investigation are high dimensionsional.
A traditional PCA-biplot depicts both the variables and the samples in the same plot
and if the dataset under consideration is high dimensional such a joint variable/sample plot can easily become
overloaded and hard to interpret.
In the present work we give a presentation of 
the linear algebra of PCA 
accentuating the natural inherent duality of the 
underlying singular value decomposition. 
In addition we point out how the basic algorithms 
easily can be adapted to produce nonlinear versions of PCA, so called multidimensional scaling, 
and we illustrate how these different versions of PCA can reveal relevant structure 
in high dimensional and complex real world datasets. 
Whether an observed
structure is relevant or not will be 
judged by knowledge supported and statistical evaluations.

Many present day datasets, coming from the application fields mentioned above,
share the statistically challenging peculiarity that 
the number of measured variables ($p$) can be very large ($ 10^4 \leq p \leq 10^{10}$), 
while at the same time the number of observations ($N$) 
sometimes can be considerably smaller ($ 10^1 \leq N \leq 10^3$).
In fact all our example datasets will share this so called 
"large $p$ small $N$" characteristic and 
our exploratory scheme, in particular the
statistical evaluation, is well adapted to cover also this situation. 
In traditional statistics one usually is presented with the reverse situation, 
i.e. "large $N$ small $p$",
and if one tries to apply traditional statistical methods 
to "large $p$ small $N$" datasets one sometimes runs into difficulties.
To begin with, in the applications we have in mind here, the underlying probability distributions 
are often unknown and then, if the number of observations is relatively small, they are 
consequently hard to estimate. 
This makes robustness of employed statistical methods a key issue. 
Even in cases when we assume that we know the underlying probability distributions
or when we use very robust statistical methods the "large $p$ small $N$" case presents difficulties.
One focus of statistical research during the last few decades has in fact been driven by 
these "large $p$ small $N$" datasets and the possibility for fast implementations 
of statistical and mathematical algorithms. 
An important example of these new trends in statistics 
is multiple hypothesis testing on a huge number of variables. 
High dimensional multiple hypothesis testing has stimulated the creation of new statistical tools 
such as the replacement of the standard concept of p-value in hypothesis testing 
with the corresponding q-value connected with the notion of false discovery rate, 
see  \cite{BH1}, \cite{BH2}, \cite{STO}, \cite{STT} for the seminal ideas. 
As a remark we point out that
multivariate statistical analogues of classical univariate statistical tests 
sometimes can perform better in multiple hypothesis testing,
but then a relatively small number of samples normally makes it necessary to first 
reduce the dimensionality of the data, for instance by using PCA, 
in order to be able to apply the multivariate tests, 
see e.g. \cite{BEY}, \cite{SWK}, \cite{KIM} for ideas in this direction. 
In the present work we give an overview and an
introduction to the above mentioned statistical notions.

The present work is in general meant to be one introduction to,
and help to stimulate more research within, the field of data exploration.
We also hope to convince the reader that statistical and 
knowledge supported visualization already is a versatile and powerful tool
for the exploration of high dimensional real world datasets.
Finally, "Knowledge supported" should here be interpreted as "any use of some extra information 
concerning a given dataset that the researcher might 
possess, have access to or gain during the exploration" when 
analyzing the visualization. 
We illustrate this knowledge supported approach 
by using knowledge based annotations coming with our example datasets.
We also briefly comment on how to use information 
collected from available databases to evaluate or preselect 
groups of significant variables, see e.g. \cite{Barry},\cite{Chen}, \cite{Khatri},\cite{Rivals} 
for some more far reaching suggestions in this direction.

\section{Singular value decomposition and Principal component analysis.}
Singular value decomposition (SVD) was discovered independently by several mathematicians 
towards the end of the 19'th century. See \cite{Stew} for an account of the early history of SVD.
Principal component analysis (PCA) for data analysis was then introduced by Pearson \cite{P} 
in 1901 and independently later developed by Hotelling \cite{H2}. 
The central idea in classical PCA is to use an SVD on the 
column averaged sample matrix to reduce the dimensionality in the data set 
while retaining as much variance as possible. 
PCA is also closely related to the Karhunen-Lo{\`e}ve expansion (KLE) of a stochastic process \cite{Kar}, \cite{Loeve}.
The KLE of a given centered stochastic process is an orthonormal $L^2$-expansion of the process 
with coefficients that are uncorrelated random variables. 
PCA corresponds to the empirical or sample version of the KLE, i.e. when the expansion is inferred from samples.
Noteworthy here is the Karhunen--Lo{\`e}ve theorem stating that if the underlying process is Gaussian, then
the coefficients in the KLE will be independent and normally distributed.
This is e.g. the basis for showing results concerning 
the optimality of KLE for filtering out Gaussian white noise.

PCA was proposed as a method to analyze genomewide expression data by Alter et al. \cite{OA}
and has since then become a standard tool in the field.
Supervised PCA was suggested by Bair et.al. as a regression and prediction method 
for genomewide data \cite{B1}, \cite{B2}, \cite{B3}.
Supervised PCA is similar to normal PCA, the only difference being that the researcher preconditions the data
by using some kind of external information. 
This external information can come from e.g. a regression analysis with respect to some
response variable or from some knowledge based considerations.  
We will here give an introduction to SVD and PCA that focus on visualization and 
the notion of using separate but synchronized biplots, i.e. plots of both
samples and variables. 
Biplots displaying samples and variables in the same usually twodimensional
diagram have been used frequently in 
many different types of applications, see e.g. \cite{GabKR1}, \cite{GabKR2}, \cite{GoHa} and \cite{Braak},
but the use of separate but synchronized biplots that we present is not standard.
We finally describe the method of multidimensional scaling which builds on
standard PCA, but we start by describing SVD 
for linear operators between finite dimensional euclidean spaces with a special focus on duality.  

\subsection{Dual singular value decomposition}
Singular value decomposition is a decomposition of 
a linear mapping between euclidean spaces.
We will discuss the finite dimensional case and we consider 
a given linear mapping $L: {\bf R}^N \longrightarrow {\bf R}^p$.

Let ${\bf e_1}, {\bf e_2}, \dots , {\bf e_N}$ be the canonical basis in ${\bf R}^N$ and let 
${\bf f_1}, {\bf f_2}, \dots , {\bf f_p}$ be the canonical basis in ${\bf R}^p$.
We regard ${\bf R}^N$ and ${\bf R}^p$ as euclidean spaces equipped with their respective canonical 
scalar products, $(\cdot, \cdot)_{{\bf R}^N}$ and $(\cdot, \cdot)_{{\bf R}^p}$,
in which the canonical bases are orthonormal.

Let $L^*:{\bf R^p} \longrightarrow {\bf R}^N$ denote the adjoint operator of $L$ defined by
\begin{equation}\label{eq:defadjoint}
(L({\bf u}), {\bf v})_{{\bf R}^p} = ({\bf u}, L^*({\bf v}))_{{\bf R}^N}  \quad ; \; {\bf u} \in {\bf R}^N
\; ; \; {\bf v} \in {\bf R}^p \, .
\end{equation}
Observe that in applications $L({\bf e_k})$, $k=1,2, \dots , N$, normally represent the arrays of observed variable values for the different samples and that
$L^*({\bf f_j})$, $j=1,2, \dots , p$, then represent the observed values of the variables.
In our example data sets, the unique $p \times N$ matrix $X$ representing $L$ in the canonical bases, i.e.
$$
 X_{jk}= ({\bf f_j}, L({\bf e}_k))_{\bf R^p}  \quad ; j=1,2, \dots , p \; ; \; k=1,2, \dots, N \,  ,
$$
contains measurements for all variables in all samples.
The transposed $N \times p$  matrix  $X^T$ contains the same information and 
represents the linear mapping $L^*$ in the canonical bases.

The goal of a dual SVD is to find orthonormal bases in 
${\bf R}^N$ and ${\bf R}^p$ such that the matrices representing the linear operators $L$ and $L^*$
have particularly simple forms.

We start by noting that directly from (\ref{eq:defadjoint}) 
we get the following direct sum decompositions into orthogonal subspaces
$$
{\bf R}^N = Ker \,  L \oplus Im \,  L^*
$$
(where $ Ker \, L$ denotes the kernel of $L$ and $Im \,  L^*$ denotes the image of $L^*$) and 
$$
{\bf R}^p = Im \, L \oplus Ker \, L^*\, .
$$
We will now make a further dual decomposition of $Im \, L$ and $Im \, L^*$.

Let $r$ denote the rank of $L:{\bf R}^N \longrightarrow {\bf R}^p$, 
i.e. $r= dim \,( Im \, L )= dim \, (Im \, L^*)$. 
The rank of the positive and selfadjoint operator $L^* \circ L:{\bf R}^N \longrightarrow {\bf R}^N$ is then also equal to $r$, and 
by the spectral theorem there exist values 
$\lambda_1 \geq \lambda_2 \geq \dots \geq \lambda_r>0$ and corresponding orthonormal 
vectors
${\bf u^1}, {\bf u^2}, \dots ,{\bf u^r}$, with ${\bf u^k} \in {\bf R}^N$, such that
\begin{equation}\label{eq:L*Leigenvectors}
L^* \circ L ({\bf u^k}) = \lambda_k^2 {\bf u^k} \quad ; \; k = 1,2, \dots , r \, .
\end{equation}
If $r <N$, i.e. $dim \, (Ker \, L)>0$, 
then zero is also an eigenvalue for $ L^* \circ L : {\bf R}^N \longrightarrow {\bf R}^N$ 
with multiplicity $N-r$. 

Using the orthonormal set of eigenvectors $\{ {\bf u^1}, {\bf u^2}, \dots ,{\bf u^r}\}$ for $L^* \circ L$ spanning $Im \, L^*$,
we define a corresponding 
set of dual vectors ${\bf v^1}, {\bf v^2}, \dots ,{\bf v^r}$ in ${\bf R}^p$ by
\begin{equation}\label{eq:utov}
L({\bf u^k})=: \lambda_k {\bf v^k} \quad ; \; k=1,2, \dots , r \, .
\end{equation}
From (\ref{eq:L*Leigenvectors}) it follows that
\begin{equation}\label{eq:vtou}
L^*({\bf v^k})= \lambda_k {\bf u^k} \quad ; \; k=1,2, \dots , r 
\end{equation}
and that
\begin{equation}\label{eq:LL*eigenvectors}
L \circ L^* ({\bf v^k}) = \lambda_k^2 {\bf v^k} \quad ; \; k = 1,2, \dots , r \, .
\end{equation}
The set of vectors  $\{ {\bf v^1}, {\bf v^2}, \dots ,{\bf v^r}\}$ defined by (\ref{eq:utov}) spans $Im \, L$ and is 
an orthonormal set of eigenvectors for the 
selfadjoint operator $L \circ L^*:{\bf R}^p \longrightarrow {\bf R}^p$.
We thus have a completely dual setup and canonical decompositions of both ${\bf R}^N$ 
and ${\bf R}^p$ into direct sums of 
subspaces spanned by 
eigenvectors corresponding to the distinct eigenvalues.
We make the following definition.
\begin{definition}
A dual  
singular value decomposition system for
an operator pair ($L$, $L^*$) is a system consisting of numbers $\lambda_1 \geq \lambda_2 \geq \dots \geq \lambda_r> 0$
and two sets of orthonormal vectors, 
$\{ {\bf u^1}, {\bf u^2}, \dots ,{\bf u^r}\}$ and $\{ {\bf v^1}, {\bf v^2}, \dots ,{\bf v^r}\}$ with
$r=rank \, (L) =rank \,(L^*)$, satisfying 
(\ref{eq:L*Leigenvectors})--(\ref{eq:LL*eigenvectors}) above.

The positive values $\lambda_1, \lambda_2, \dots, \lambda_r$ are called {\it the singular values} of ($L$, $L^*$).
We will call the vectors ${\bf u^1}, {\bf u^2}, \dots ,{\bf u^r}$ principal components for $Im \, L^*$
and the vectors ${\bf v^1}, {\bf v^2}, \dots ,{\bf v^r}$ principal components for $Im \, L$.
\end{definition}
Given a dual SVD system we now complement the principal components for $Im \, L^*$, ${\bf u^1}, {\bf u^2}, \dots ,{\bf u^r}$, 
to an orthonormal basis ${\bf u^1}, {\bf u^2}, \dots ,{\bf u^N}$ in ${\bf R}^N$ and 
the principal components for $Im\, L$, ${\bf v^1}, {\bf v^2}, \dots ,{\bf v^r}$, 
to an orthonormal basis ${\bf v^1}, {\bf v^2}, \dots ,{\bf v^p}$ in ${\bf R}^p$.

In these bases we have that
\begin{equation}\label{eq:diagform}
({\bf v^j}, L({\bf u^k}))_{{\bf R}^p} =(L^* ({\bf v^j}), {\bf u^k})_{{\bf R}^N} =
\begin{cases}
\lambda_k \delta_{jk} &  \text{if $j, k \leq r$} \\
0 & \text{otherwise} \, .
\end{cases}
\end{equation}
This means that in these ON-bases $L: {\bf R}^N \longrightarrow {\bf R}^p$ 
is represented by the diagonal $p \times N$ matrix
\begin{gather}\label{eq:diagmatrix}
\begin{bmatrix} 
D& 0 \\
0 & 0
\end{bmatrix}
\end{gather}
where $D$ is the $r \times r$ diagonal matrix having the singular values of ($L$, $L^*$) in descending order on the diagonal.
The adjoint operator $L^*$ 
is represented in the same bases by the transposed matrix, 
i.e. a diagonal $N \times p$ matrix. 

We translate this to operations on the corresponding matrices as follows.
Let $U$ denote the $N \times r$ matrix having the coordinates, in the canonical basis in ${\bf R^N}$, 
for ${\bf u^1}, {\bf u^2}, \dots {\bf u^r}$
as columns, and let $V$ denote the $p \times r$ matrix having the coordinates, in the canonical basis in ${\bf R^p}$, for 
${\bf v^1}, {\bf v^2}, \dots {\bf v^r}$
as columns. Then (\ref{eq:diagform}) is equivalent to
$$
X=VDU^T \quad \mbox{and} \quad X^T=U D V^T \, .
$$
This is called a dual singular value decomposition for the pair of matrices ($X$, $X^T$).

Notice that the singular values  
and the corresponding separate eigenspaces for $L^* \circ L$  as described above are canonical, 
but that the set $\{ {\bf u^1}, {\bf u^2}, \dots ,{\bf u^r} \}$ (and thus also the connected set
$\{ {\bf v^1}, {\bf v^2}, \dots ,{\bf v^r}\}$) is not 
canonically defined by $L^* \circ L$. 
This set is only canonically defined up to 
actions of the appropriate orthogonal groups on the separate eigenspaces.

\subsection{Dual principal component analysis}
We will now discuss how to use a dual SVD system to obtain optimal approximations of
a given operator $L:{\bf R}^N \longrightarrow {\bf R^p}$ by operators of lower rank. 
If our goal is to visualize the data, then it is natural to measure the approximation error
using a unitarily invariant norm, i.e. a norm $\| \cdot \|$ that is invariant with respect to unitary transformations
on the variables or on the samples, i.e.
\begin{equation}\label{eq:unitinv}
\| L\| = \| V \circ L \circ U \| \quad \mbox{for all $V$ and $U$ s.t.} \; V^*V = Id \; \mbox{and} \; U^* U= Id \, .
\end{equation}
Using an SVD, directly from (\ref{eq:unitinv}) we conclude 
that such a norm is necessarily a symmetric function of the
singular values of the operator. 
We will present results for the $L^2$--norm of the singular values, but 
the results concerning optimal approximations are actually valid with respect to any 
unitarily invariant norm, see e.g. \cite{MIRS} and \cite{RAO} for information in this direction.
We omit proofs, but all the results in this section are 
proved using SVDs for the involved operators.

The  Frobenius (or Hilbert-Schmidt) norm for an operator $L:{\bf R}^N \longrightarrow {\bf R^p}$ of rank $r$ 
is defined by
$$
\| L \|_{F}^2:= \sum_{k=1}^r \lambda^2_k \, ,
$$
where $\lambda_k$, $k=1,2, \dots , r$ are the singular values of ($L$, $L^*$).

Now let $\mathcal{M}_{n \times n}$ denote the
set of real $n \times n$ matrices. 
We then define the set of orthogonal projections in ${\bf R}^n$ of rank $s \leq n$ as
$$
\mathcal{P}^n_s := \{ \Pi \in \mathcal{M}_{n \times n} \; ; \; \Pi^*= \Pi\; ; \; \Pi \circ \Pi = \Pi \; ; \;  rank (\Pi) =s \} \, .
$$
One important thing about orthogonal projections is that they never increase the Frobenius norm, i.e.
\begin{lemma}
Let $L:{\bf R}^N \longrightarrow {\bf R}^p$ be a given linear operator. Then
$$
\| \Pi \circ L \|_{F} \leq \| L \|_{F} \quad \mbox{for all} \; \Pi \in \mathcal{P}^p_s 
$$
and
$$
\| L \circ \Pi \|_{F} \leq \| L \|_{F} \quad \mbox{for all} \; \Pi \in \mathcal{P}^N_s \, .
$$
\end{lemma}
Using this Lemma one can prove the following approximation theorems.
\begin{theorem}\label{th:approxmax}
Let $L:{\bf R}^N \longrightarrow {\bf R}^p$ be a given linear operator. Then
\begin{eqnarray}\label{eq:optapproxmax}
\sup_{\Pi^p \in\mathcal{P}^p_s\; ; \;\Pi^N \in\mathcal{P}^N_s } \| {\Pi}^p \circ L \circ {\Pi}^N \|_F =
\sup_{\Pi \in\mathcal{P}^p_s} \| \Pi \circ L \|_{F} = \nonumber \\
= \sup_{\Pi \in\mathcal{P}^N_s} \| L \circ \Pi \|_{F}= \left\{ \sum_{k=1}^{\min(s,r)} \lambda_k^2 \right\}^{1/2}
\end{eqnarray}
and equality is attained in (\ref{eq:optapproxmax}) by projecting onto the $\min(s,r)$ first
principal components for $ Im \, L$ and $Im \, L^*$.
\end{theorem}
\begin{theorem}\label{th:approxmin}
Let $L:{\bf R}^N \longrightarrow {\bf R}^p$ be a given linear operator. Then
\begin{eqnarray}\label{eq:optapproxmin}
\inf_{\Pi^p \in\mathcal{P}^p_s\; ; \;\Pi^N \in\mathcal{P}^N_s } \| L- {\Pi}^p \circ L \circ {\Pi}^N \|_F =
\inf_{\Pi \in\mathcal{P}^p_s} \| L- \Pi \circ L \|_{F} = \nonumber \\
= \inf_{\Pi \in\mathcal{P}^N_s} \| L- L \circ \Pi \|_{F}= \left\{ \sum_{k=\min(s,r)+1}^{\max(s,r)} \lambda_k^2 \right\}^{1/2}
\end{eqnarray}
and equality is attained in (\ref{eq:optapproxmin}) by projecting onto the $\min(s,r)$ first 
principal components for $ Im \, L$ and $Im \, L^*$.
\end{theorem}
We loosely state these results as follows.

\noindent 
{\bf Projection dictum}

\noindent
{\it Projecting onto a set of first principal components maximizes average projected vector length and also
minimizes average projection error.}

We will briefly discuss interpretation for applications. 
In fact in applications the representation of our linear mapping $L$ normally has a 
specific interpretation in the original canonical bases. 
Assume that $L({\bf e_k})$, $k=1,2, \dots , N$ represent samples and that
$L^*({\bf f_j})$, $j=1,2, \dots , p$ represents variables.
To begin with, if the samples are centered, i.e.
$$
\sum_{k=1}^N L({\bf e_k}) =0 \, ,
$$
then $\| L \|_F^2$ corresponds to the statistical {\it variance} of the sample set. 
The basic projection dictum 
can thus be restated for sample-centered data as follows.

\noindent
{\bf Projection dictum for sample-centered data}

\noindent
{\it Projecting onto a set of  first principal components maximizes 
the variance in the set of projected data points and also
minimizes average projection error.}

In applications we are also interested in keeping track of the value
 \begin{equation}\label{eq:SVD}
X_{jk}= ({\bf f_j}, L( {\bf e_k})) \, .
\end{equation}
It represents the $j$'th variable's value in the $k$'th sample. 

Computing in a dual SVD system for ($L$, $L^*$) in (\ref{eq:SVD}) we get
\begin{equation}\label{eq:Xjkestimate}
X_{jk} = \lambda_1 ({\bf e_k}, {\bf u^1})({\bf f_j}, {\bf v^1}) + \cdots + \lambda_r ({\bf e_k}, {\bf u^r})({\bf f_j}, {\bf v^r}) \, .
\end{equation}
Now using (\ref{eq:utov}) and (\ref{eq:vtou}) we conclude that
$$
X_{jk} = \frac{1}{\lambda_1} ({\bf e_k}, L^* ({\bf v^1}))({\bf f_j}, L ({\bf u^1})) + \cdots + \frac{1}{\lambda_r} ({\bf e_k}, L^* ({\bf v^r}))({\bf f_j}, L ({\bf  u^r})) \, .
$$
Finally this implies the fundamental {\bf biplot formula}
\begin{equation}\label{eq:synchSVD}
X_{jk} = \frac{1}{\lambda_1} (L ({\bf e_k}), {\bf v^1})(L^* ({\bf f_j}), {\bf u^1}) + \cdots + \frac{1}{\lambda_r} (L ({\bf e_k}), {\bf v^r})(L^* ({\bf f_j}), {\bf u^r}) \, .
\end{equation}
We now introduce the following scalar product in ${\bf R}^r$
$$
({\bf a}, {\bf b})_{\lambda} := \frac{1}{\lambda_1} a_1 b_1 + \cdots +\frac{1}{\lambda_r} a_r b_r \quad ; \; {\bf a} , {\bf b} \in {\bf R}^r \, .
$$
Equation (\ref{eq:synchSVD}) thus means that if we express the sample vectors in the basis ${\bf v^1}, {\bf v^2}, \dots , {\bf v^r}$ 
for $Im \, L$ and 
the variable vectors in the basis
${\bf u^1}, {\bf u^2}, \dots , {\bf u^r}$ for $Im \, L^*$ , 
then we get the value of $X_{jk}$ simply by taking the $(\cdot, \cdot)_{\lambda}$-scalar product in ${\bf R}^r$  between the coordinate 
sequence for the k'th sample  and the coordinate sequence for the j'th variable. 

This means that  if we work in a synchronized way in ${\bf R}^r$ 
with the coordinates for the samples (with respect to the basis ${\bf v^1}, {\bf v^2}, \dots , {\bf v^r}$) and
with the coordinates for the variables (with respect to the basis ${\bf u^1}, {\bf u^2}, \dots , {\bf u^r}$)  then
the {\bf relative positions} of the coordinate sequence for a variable and the coordinate sequence for a sample in ${\bf R}^r$  have a very precise meaning
given by (\ref{eq:synchSVD}).

Now let $S \subset \{ 1,2, \dots , r \}$ be a subset of indices and let $|S|$ denote the number of elements in $S$. Then let
$\Pi^p_S: {\bf R}^p \longrightarrow {\bf R}^p$ be the orthogonal projection onto the subspace spanned by the
principal components for $Im \, L$ whose indices belong to $S$.
In the same way let $\Pi^N_S: {\bf R}^N \longrightarrow {\bf R}^N$ be the orthogonal projection onto the subspace spanned by the
principal components for $ Im \, L^*$ whose indices belong to $S$.

If $L({\bf e_k})$, $k=1,2, \dots , N$ represent samples, then $\Pi^p_S \circ L({\bf e_k})$, 
$k=1,2, \dots , N$,
represent {\it $S$-approximative samples}, 
and correspondingly if  $L^*({\bf f^j})$, $j=1,2, \dots , p$, represent variables then 
$\Pi_S^N \circ L^* ({\bf f_j})$, $j=1,2, \dots , p$,
represent {\it $S$-approximative variables}.

We will interpret the matrix element
\begin{equation}\label{eq:synchSVDapprox}
X^S_{jk} : = ({\bf f_j}, \Pi_S^p \circ L ({\bf e_k})) 
\end{equation}
as representing  the $j$'th $S$-approximative variable's 
value in the \linebreak $k$'th $S$-approximative sample.

By the biplot formula (\ref{eq:synchSVD}) for the operator $\Pi_S^p \circ L$ we actually have
\begin{equation}\label{eq:synchSVDappr}
X^S_{jk} = \sum_{m \in S} \frac{1}{\lambda_m} (L ({\bf e_k}), {\bf v^m})(L^* ({\bf f_j}), {\bf u^m}) \, .
\end{equation}
If $|S| \leq 3$ we can visualize 
our approximative samples and approximative variables
working in a synchronized way in ${\bf R}^{|S|}$ 
with the coordinates for the approximative samples and
with the coordinates for the approximative variables.
The {\bf relative positions} of the coordinate sequence for an approximative variable 
and the coordinate sequence for an approximative sample in ${\bf R}^{|S|}$  then have the very precise meaning
given by (\ref{eq:synchSVDappr}).

Naturally the information content of a biplot visualization depends in a crucial way on the approximation error we make.
The following result gives the basic error estimates.
\begin{theorem}\label{th:approxerror}
With notations as above we have the following projection error estimates 
\begin{eqnarray}
 \sum_{j=1}^p \sum_{k=1}^N |X_{jk} - X^S_{jk}|^2 = \sum_{ i \notin S} | \lambda_i |^2  \\
\sup_{j=1, \dots ,p \; ; \; k=1, \dots , N\,} | X_{jk} - X^S_{jk}| \leq \sup_{ i \notin S} |\lambda_i| \,.
\end{eqnarray}
\end{theorem}
We will use the following statistics for measuring {\it projection content}: 
\begin{definition}
With notations as above, the {\it $L^2$-projection content} connected with
the subset $S$ is by definition
$$
\alpha_2(S) := \frac{\sum_{ i \in S} | \lambda_i |^2 }{\sum_{ i=1}^r | \lambda_i |^2 }\, .
$$

\end{definition}
We note that, in the case when we have sample centered data, $\alpha_2(S)$ is precisely the quotient between 
the amount of variance that
we have "captured" in our projection and the total variance. 
In particular if $\alpha_2(S)=1$
then we have captured all the variance. Theorem \ref{th:approxerror} shows that
we would like to have good control of the distributions of eigenvalues for general
covariance matrices. We will address this issue for random matrices below, but we already 
here point out that we will estimate projection information content, or the signal to noise ratio, in a projection of real world data by comparing the observed $L^2$-projection content and the $L^2$-projection contents for 
corresponding randomized data.

\subsection{Nonlinear PCA and multidimensional scaling}

We begin our presentation of multidimensional scaling 
by looking at  the reconstruction problem,
i.e. how to reconstruct a dataset 
given only a proposed covariance or distance matrix.
In the case of a covariance matrix, 
the basic idea is to try to factor a corresponding sample centered SVD or¨
slightly rephrased by taking the square root of the covariance matrix.

Once we have established a reconstruction scheme we note that we can apply it to any 
proposed "covariance" or "distance" matrix, as long as they have the correct 
structure, even if they are artificial and a priori are not constructed 
using euclidean transformations on an existing data matrix. 
This opens up the possibility for using "any type" of similarity measures
between samples or variables to construct artificial covariance or distance matrices.

We consider a $p \times N$ matrix $X$ where the $N$ columns $\{{\bf x_1}, \dots , {\bf x_N}\}$ consist of values of
measurements for $N$ samples of $p$ variables. We will throughout this section assume that $p \geq N$.
We introduce the $N \times 1$ vector
$$
{\bf 1}=[1,1,\dots, 1]^T \, ,
$$
and we recall that the $N \times N$ covariance matrix of the data matrix 
${\bf X}=[{\bf x_1}, \dots, {\bf x_N}]$ is given as
$$
{\bf C}({\bf x_1}, \dots, {\bf x_N})=({\bf X}- \frac{1}{N}{\bf X}\, {\bf 1}\, {\bf 1}^T)^T({\bf X}- \frac{1}{N}{\bf X}\, {\bf 1}\,{\bf 1}^T) \, .
$$
We will also need the (squared) distance matrix defined by
$$
{\bf D}_{j k} ({\bf x_1}, \dots, {\bf x_N}):= |{\bf x_j}-{\bf x_k}|^2 \quad ; \quad j,k=1,2, \dots, N \, .
$$

We will now consider the problem of reconstructing a data matrix $X$ given only the 
corresponding covariance matrix or the corresponding distance matrix.
We first note that since the covariance and the distance matrix
of a data matrix $X$ both are invariant under euclidean transformations in 
${\bf R}^p$ of the columns of $X$,
it will,  if at all, only be possible to reconstruct the $p \times N$ matrix $X$ modulo
euclidean transformations in ${\bf R^p}$ of its columns.

We next note that the matrices ${\bf C}$ and ${\bf D}$ are connected. In fact we have
\begin{proposition}\label{prop:cvsd}
Given data points ${\bf x_1}, {\bf x_2}, \dots, {\bf x_N}$ in ${\bf R}^p$ and the corresponding 
covariance and distance matrices, ${\bf C}$ and ${\bf D}$, we have that
\begin{equation}\label{eq:dvscmatrix}
{\bf D}_{jk}= {\bf C}_{jj} + {\bf C}_{kk} -2 {\bf C}_{jk} \, .
\end{equation}
Furthermore
\begin{equation}\label{eq:cvsdmatrix}
{\bf C}_{jk}= \frac{1}{2N} \sum_{i=1}^N ({\bf D}_{ij}+ {\bf D}_{ik})-\frac{1}{2}{\bf D}_{jk} - 
\frac{1}{2N^2} \sum_{i,m
=1}^N {\bf D}_{im} \, ,
\end{equation}
or in matrixnotation,
\begin{equation}\label{eq:matrixnotation}
{\bf C}= -\frac{1}{2} \left(I- \frac{{\bf 1}{\bf 1}^T}{N}\right)\,  {\bf D} \,  
\left(I- \frac{{\bf 1}{\bf 1}^T}{N}\right) \, .
\end{equation}
\end{proposition}

\begin{proof}
Let
$$
{\bf y_i}:= {\bf x_i} - \frac{1}{N}\sum_{j=1}^N {\bf x_j} \quad ; \; i=1,2, \dots , N \, .
$$
Note that 
$$
{\bf C}_{jk}= {\bf y_j^T} \, {\bf y_k} \quad \mbox{and} \quad {\bf D}_{jk}=|{\bf y_j}-{\bf y_k}|^2 \, ,
$$
and that $\sum_{j=1}^N {\bf y_j} =0$.

Equality (\ref{eq:dvscmatrix}) above is simply the polarity condition
\begin{equation}\label{eq:polar}
{\bf D}_{jk}= |{\bf y_j}|^2 + |{\bf y_k}|^2 -2 {\bf C}_{jk} \, .
\end{equation}
Moreover, since
$$
\sum_{j=1}^N {\bf C}_{jk} = 0 \quad \mbox{and} \quad \sum_{k=1}^N {\bf C}_{jk} = 0 \, ,
$$
by summing over both $j$ and $k$ in (\ref{eq:polar}) above we get
\begin{equation}\label{eq:doublesum}
\sum_{j=1}^N |{\bf y_j}|^2 = \frac{1}{2N} \sum_{j,k =1}^N {\bf D}_{jk} \, .
\end{equation}
On the other hand, by summing only over $j$ in (\ref{eq:polar}) we get
\begin{equation}\label{eq:sum}
\sum_{j=1}^N {\bf D}_{jk}= N |{\bf y_k}|^2 + \sum_{j=1}^N |{\bf y_j}|^2 \, .
\end{equation}
Combining (\ref{eq:doublesum}) and (\ref{eq:sum}) we get
$$
|{\bf y_k}|^2 = \frac{1}{N} \sum_{j=1}^N {\bf D}_{jk} - \frac{1}{2N^2}\sum_{j,k =1}^N {\bf D}_{jk} \, .
$$
Plugging this into (\ref{eq:polar}) we finally conclude that
$$
{\bf D}_{jk}= \frac{1}{N} \sum _{i=1}^N ({\bf D}_{ij} + {\bf D}_{ik}) - \frac{1}{N^2} \sum_{i,j=1}^N {\bf D}_{ij}
-2 {\bf C}_{jk} \, .
$$
This is (\ref{eq:cvsdmatrix}).
\end{proof}

Now let $\mathcal{M}_{N \times N}$ denote the set of all real $N  \times N$ matrices.
To reconstruct a $p \times N$ data matrix $X= [{\bf x_1}, \dots , {\bf x_N}]$ 
from a given $N \times N$ covariance or $N \times N$ distance matrix amounts to invert the 
mappings:
$$
\Phi :{\bf R}^p \times \cdots \times {\bf R}^p \ni ({\bf x_1}, {\bf x_2}, \dots, {\bf x_N}) 
\mapsto {\bf C}({\bf x_1}, {\bf x_2}, \dots , {\bf x_N}) \in \mathcal{M}_{N \times N}, 
$$
and 
$$
\Psi: {\bf R}^p \times \cdots \times {\bf R}^p \ni ({\bf x_1}, {\bf x_2}, \dots, {\bf x_N}) 
\mapsto {\bf D}({\bf x_1},{\bf x_2}, \dots , {\bf x_N})\in \mathcal{M}_{N \times N} \, .
$$
In general it is of course impossible to invert these mappings since 
both $\Phi$ and $\Psi$ are far from surjectivity and injectivity.

Concerning injectivity, it is clear that both $\Phi$ and $\Psi$ are invariant under the euclidean group $E(p)$ acting on 
${\bf R}^p \times \cdots \times {\bf R}^p$, i.e.
under transformations 
$$({\bf x_1}, {\bf x_2}, \dots , {\bf x_N}) \mapsto (S {\bf x_1} +{\bf b}, S {\bf x_2} + {\bf b}, \dots , S {\bf x_n} + {\bf b}) \; ,$$ 
where $ {\bf b} \in {\bf R}^p$ and $S \in O(p)$.

This makes it natural to introduce the quotient manifold
$$ ({\bf R}^p \times \cdots \times {\bf R}^p)/ E(p)$$ 
 and possible to define the induced mappings $\overline{{\Phi}}$ and $\overline{{\Psi}}$,
well defined on the equivalence classes and factoring the mappings $\Phi$ and $\Psi$ by the quotient mapping. 
We will write
$$
\overline{{\Phi}}:([{\bf x_1}, {\bf x_2}, \dots, {\bf x_N}]) \mapsto {\bf C}({\bf x_1}, {\bf x_2}, \dots , {\bf x_N}) 
$$
and 
$$
\overline{{\Psi}}: ([{\bf x_1}, {\bf x_2}, \dots, {\bf x_N}]) \mapsto {\bf D}({\bf x_1},{\bf x_2}, \dots , {\bf x_N}) \, .
$$
We shall show below that both $\overline{\Phi}$ and $\overline{\Psi}$ are injective.

Concerning surjectivity of the maps $\Phi$ and $\Psi$, 
or $\overline{{\Phi}}$ and $\overline{{\Psi}}$, we will first describe the image
of $\Phi$. Since the images of $\Phi$ and $\Psi$ are connected through
Proposition \ref{prop:cvsd} above this implicitly describes the image also for $\Psi$.
It is theoretically important that both these image sets turn 
out to be closed and convex subsets of $\mathcal{M}_{N \times N}$.
In fact we claim that the following set is the image set of $\overline{{\Phi}}$:
$$\mathcal{P}_{N \times N} := 
\left\{ A \in \mathcal{M}_{N \times N} \; ; \; A^T=A 
\; ,\; A \geq 0 \; , \; A\, {\bf 1}= 0 \right\} \, .
$$
To begin with it is clear that the image of $\Phi$ is equal to the image of $\overline{{\Phi}}$ and that it is
included in $\mathcal{P}_{N \times N}$, i.e.
$$
({\bf R}^p \times \cdots {\bf R}^p)/ E(p) \ni ([{\bf x_1}, {\bf x_2}, \dots, {\bf x_N}]) 
\mapsto {\bf C}({\bf x_1},{\bf x_2}, \dots , {\bf x_N}) \in \mathcal{P}_{N \times N} \, .
$$
The following proposition implies that $\mathcal{P}_{N \times N}$ is the image set of $\overline{{\Phi}}$
and it is the main result of this subsection.
\begin{proposition}
The mapping $\overline{{\Phi}}$:
$$
({\bf R}^p \times \cdots \times {\bf R}^p)/ E(p) \ni ([{\bf x_1}, {\bf x_2}, \dots, {\bf x_N}]) 
\mapsto {\bf C}({\bf x_1},{\bf x_2}, \dots , {\bf x_N}) \in \mathcal{P}_{N \times N} \, 
$$
is a bijection.
\end{proposition}
\begin{proof}
If ${\bf A} \in \mathcal{P}_{N \times N}$ we can, by the spectral theorem, 
find a unique symmetric and positive $N \times N$ matrix
${\bf B}=[{\bf b}_1, {\bf b}_2 , \dots , {\bf b}_N]$ (the square root of ${\bf A}$) with $rank({\bf B})= rank({\bf A})$ 
such that ${\bf B}^2={\bf A}$ and ${\bf B}\, {\bf 1}=0$.

We now map the points $({\bf b}_1, {\bf b}_2 , \dots , {\bf b}_N)$ laying in ${\bf R}^N$ isometrically to points 
$({\bf x}_1,{\bf x}_2, \dots , {\bf x}_N)$ in ${\bf R}^p$. 
This is trivially possible since $p \geq N$.
The corresponding covariance matrix ${\bf C}({\bf x}_1, \dots , {\bf x}_N)$ will be equal to ${\bf A}$.
This proves surjectivity. That the mapping  $\overline{{\Phi}}$ is injective follows directly from the following lemma.
\begin{lemma}
Let $\{ {\bf y}_k\}_{k=1}^N$ and $\{\tilde{\bf y}_k\}_{k=1}^N$  be two sets of vectors in ${\bf R}^p$. 
If 
$$
{\bf y}^T_k {\bf y}_j = \tilde{\bf y}^T_k \tilde{\bf y}_j \quad \mbox{for} \; j,k=1,2, \dots , N \, ,
$$
then there exists an ${\bf S} \in {\bf O}(p)$ such that
$$
{\bf S} ({\bf y}_k) = \tilde{\bf y}_k \quad \mbox{for} \; k=1,2, \dots , N \, .
$$
\end{lemma}
\begin{proof}
Use the Gram-Schmidt orthogonalization procedure on both sets at the same time.
\end{proof}
\end{proof}

We will in our explorations of high dimensional real world datasets 
below use  "artificial" distance matrices constructed from 
geodesic distances on carefully created graphs connecting the samples or the variables.
These distance matrices are converted to unique corresponding covariance matrices
which in turn, as described above, 
give rise to canonical elements in  
$({\bf R}^p \times \cdots \times {\bf R}^p)/ E(p)$.
We then pick sample centered representatives on which we perform PCA.
In this way we can visualize low dimensional "approximative graph distances" in the dataset. 
Using graphs in the sample set constructed from a $k$ nearest neighbors or a locally euclidean approximation procedure,
 this approach corresponds to the ISOMAP algorithm 
introduced by Tenenbaum et. al. \cite{Tenen}. 
The ISOMAP algorithm can as we will see below be very useful in the exploration of DNA microarray data, 
see Nilsson et. al. \cite{Nils}
for one of the first applications of ISOMAP in this field.

We finally remark that if a proposed 
artificial distance or covariance matrix does not have the correct structure,
i.e. if for example a proposed covariance matrix does not belong to $\mathcal{P}_{N \times N}$,
we begin by projecting the proposed covariance matrix onto 
the unique nearest point in the closed and convex set $\mathcal{P}_{N \times N}$ 
and then apply the scheme presented above to that point.

\section{The basic statistical framework.}

We will here fix some notation and for the non-statistician reader's convenience 
at the same time recapitulate some standard multivariate statistical theory.
In particular we want to stress some basic facts concerning robustness of statistical testing.

Let $\mathcal{S}$ be the sample space consisting of all possible samples (in our example datasets equal to all trials of patients)
equipped with a probability measure $P:2^{\mathcal{S}} \longrightarrow [0, +\infty]$ 
and let $X=(X_1, \dots, X_p)^T$ be a random vector from 
$\mathcal{S}$ into $\bf{R}^p$. 
The coordinate functions $X_i: \mathcal{S} \longrightarrow \bf{R}$, $i=1,2,\dots, p$ are random variables and in our example datasets they represent the expression levels of the different genes. 

We will be interested in the law of $X$, i.e. the induced probability measure $P(X^{-1}(\cdot))$ defined on the measurable 
subsets of ${\bf{R}^p}$.
If it is absolutely continuous with respect to Lebesgue measure then there exists a probability density function (pdf)
$f_X(\cdot): \bf{R}^p \longrightarrow [0, \infty) $  that belongs to $\mathcal{L}^1({\bf R}^p)$ and satisfies
\begin{equation}
P(\{s \in \mathcal{S} \; ; \; X(s) \in A\}) = \int_A f_X(x) \, dx
\end{equation}
for all events (i.e. all Lebesgue measurable subsets) $A \subset \bf{R}^p$.
This means that the pdf $f_X(\cdot)$ contains all necessary information in order to 
compute the probability that an event has occurred, i.e. that 
the values of $X$ belong to a certain given set $A \subset \bf{R}^p$.

All statistical inference procedures are concerned with trying to learn as much as possible 
about an at least partly unknown induced probability measure 
$P(X^{-1}(\cdot))$ from a given set of $N$ observations,
$\{ {\bf x^1}, {\bf x^2}, \dots , {\bf x^N} \} $ (with $ {\bf x}^i \in {\bf R}^p$ for $i=1,2, \dots , N$), of 
the underlying random vector $X$.

Often we then assume that we know something about the {\it structure} of the corresponding pdf $f_X(\cdot)$ and we try to
make statistical inferences about the {\it detailed form} of the function $f_X(\cdot)$.

The most important probability distribution in multivariate statistics is the multivariate normal distribution.
In $\bf{R}^p$ it is given by the $p$-variate 
pdf $n:\bf{R}^p \longrightarrow (0, \infty)$ where
\begin{equation}\label{eq:normalpdf}
n(x):= (2\pi)^{-p/2} | \Gamma |^{-1/2} e^{-\frac{1}{2}(x- \mu)^T \Gamma^{-1} (x- \mu)} \quad; \; x \in \bf{R}^p \, .
\end{equation}
It is characterized by 
the symmetric and positive definite $p \times p$ matrix $\Gamma$ and the $p$-column vector $\mu$, and 
$| \Gamma |$ stands for the absolute value of the determinant of $\Gamma$.
If a random vector $X:\mathcal{S} \longrightarrow \bf{R}^p$ has the $p$-variate 
normal pdf (\ref{eq:normalpdf}) we say that $X$ has the ${\bf N}({\bf \mu}, \Gamma)$ distribution.
If  $X$ has the ${\bf N}({\bf \mu}, \Gamma)$ distribution then 
{\bf the expected value} of $X$ is equal to ${\bf \mu}$ i.e. 
\begin{equation}
\mathcal{E}(X):= \int_{\mathcal{S}} X(s) \, dP= {\bf \mu} \, ,
\end{equation}
 and {\bf the covariance matrix} of $X$ is equal to $\Gamma$, i.e.
 \begin{equation}
 \mathcal{C}(X) := \int_S (X - \mathcal{E}(X)) (X- \mathcal{E}(X))^T \, dP = \Gamma \, . 
 \end{equation}
Assume now
that $X^1, X^2, \dots , X^N$ are given independent and identically distributed (i.i.d.) random vectors.
A {\bf test statistic} $\mathcal{T}$ is then a function 
$(X^1, X^2, \dots , X^N) \mapsto \mathcal{T}(X^1, X^2, \dots , X^N)$.
Two important test statistics are {\bf the sample mean} vector of a sample of size $N$ 
$$\overline{X}^N:= \frac{1}{N} \sum_{i=1}^N X^i \, ,$$
and {\bf the sample covariance matrix} of a sample of size $N$
$$S^N:= \frac{1}{N-1} \sum_{i=1}^N (X^i-\overline{X})(X^i-\overline{X})^T \, .$$

If $X^1, X^2, \dots , X^N$ are independent and ${\bf N}({\bf \mu}, \Gamma)$ distributed, then
the mean $\overline{X}^N$ has the 
${\bf N}({\bf \mu}, \frac{1}{N} \Gamma)$ distribution.
In fact this result is asymptotically robust with respect to the underlying distribution. 
This is a consequence of the well known and celebrated central limit theorem:
\begin{theorem}
 If the random $p$ vectors $X^1, X^2, X^3, \dots $ 
are independent and identically distributed with means $\mu \in \bf{R}^p$ and covariance matrices $\Gamma$, then
the limiting distribution of
$$
(N)^{1/2} \left( \overline{X}^N-\mu\right)
$$
as $N \longrightarrow \infty$ is ${\bf N}(0, \Gamma)$.
\end{theorem}
The central limit theorem tells us that, if we know nothing and still  
need to assume some structure on the underlying p.d.f., then asymptotically the $N(\mu, \Gamma)$
distribution is the only reasonable assumption. 
The distributions of different statistics are of course more or less sensitive to the underlying distribution.
In particular the standard univariate Student t-statistic, 
used to draw inferences about a univariate sample mean, 
is very robust with respect to the underlying probability distribution.
In for example the study on statistical robustness \cite{RTG} the authors conclude that:

 {\it "...the two-sample t-test is so robust that it can be recommended in nearly all applications."}
 
This is in contrast with many statistics connected with the sample covariance matrix.
A central example in multivariate analysis
is the set of eigenvalues of the sample covariance matrix. 
These statistics have a more complicated behavior.
First of all, 
if $X^1, X^2, \dots , X^N$ with values in ${\bf R}^p$ are 
independent and ${\bf N}({\bf \mu}, \Gamma)$ distributed then 
the sample covariance matrix
is said to have a Wishart distribution ${\bf W}_p(N, \Gamma)$. 
If $N >p \,$  the Wishart distribution 
is absolutely continuous with respect to Lebesgue measure and the probability density function is explicitly known,
see e.g. Theorem 7.2.2. in \cite{AN}.
If $N >> p$
then the  eigenvalues of the sample covariance matrix are good estimators 
for the corresponding eigenvalues of the underlying covariance matrix $\Gamma$, see \cite{AN1} and \cite{AN}. 
In the applications we have in mind we often have the reverse situation, i.e. $p >> N$, and then the eigenvalues for the 
sample covariance matrix are far from consistent estimators for the corresponding eigenvalues of the underlying covariance matrix.
In fact if the underlying covariance matrix is the identity matrix 
it is known (under certain growth conditions on the underlying distribution) that 
if we let $p$ depend on $N$ and if $p/N \longrightarrow \gamma \in (0, \infty)$ as $N \longrightarrow \infty$, 
then the largest eigenvalue for the sample covariance matrix tends to $ ( 1+ \sqrt{\gamma} )^2$, see e.g. \cite{YBK},
and not to $1$ as one maybe could have expected. 
This result is interesting and can be useful, but
there are many open questions concerning 
the asymptotic theory for the "large $p$, large $N$ case", in particular if we go beyond the case of 
normally distributed data, see e.g. \cite{BAI}, \cite{DIAC}, \cite{IMJ1}, \cite{IMJ2} and \cite{KARO} for an overview of the current state of the art.
To estimate  the information content or signal to noise ratio in our PCA plots
we will therefore rely mainly on randomization tests and 
not on the (not well enough developed) 
asymptotic theory for the distributions of eigenvalues of random matrices.

\section{Controlling the false discovery rate}

When we perform for example a Student t-test to estimate whether or not 
two groups of samples have the same mean value for a specific variable we are
performing a hypothesis test. 
When we do the same thing for a large number of variables at the same time
we are testing one hypothesis for each and every variable. 
It is often the case in the applications we have in mind 
that tens of thousands of features are tested at the same time 
against some null hypothesis $H_0$, e.g. that the mean values in two given groups are identical.
To account for this multiple hypotheses testing, several methods have been proposed, 
see e.g. \cite{Yin} for an overview and comparison of some existing methods.
We will give a brief review of some basic notions.

Following the seminal paper by Benjamini and Hochberg \cite{BH1},
 we introduce the following notation.
We consider the problem of testing $m$ null hypotheses $H_0$ 
against the alternative hypothesis $H_1$. 
We let $m_0$ denote the number of true nulls.
We then  let $R$ denote the total number of rejections, which we will call 
the total number of statistical discoveries, 
and let $V$ denote the 
number of false rejections. 
In addition we introduce stochastic variables $U$ and $T$ according to 
Table \ref{table:falsediscovery}.

\begin{table}[h!]
\caption{Test statistics} % title of Table
\centering % used for centering table
\begin{tabular}{c c c c} % centered columns (2 columns)
\hline\hline %inserts double horizontal lines
& Accept $H_0$ & Reject $H_0$& Total \\ [0.5ex] % inserts table
%heading
\hline % inserts single horizontal line
$H_0$ true	& U& V& $m_0$ \\
$H_1$ true	& T& S & $m_1$ \\
	& $m-R$ & R & $m$ \\
[1ex] % [1ex] adds vertical space
\hline %inserts single line
\end{tabular}
\label{table:falsediscovery} % is used to refer this table in the text
\end{table}

The false discovery rate was loosely defined by Benjamini and Hochberg as
the expected value $E(\frac{V}{R})$. 
More precisely the false discovery rate is defined as
\begin{equation}
FDR:= E(\frac{V}{R} \, | \, R > 0) \, P( R>0) \, .
\end{equation}
The false discovery rate measures the proportion of Type I errors among 
the statistical discoveries. 
Analogously we define corresponding
statistics according to Table \ref{table:discoveryrates}.
\begin{table}[h!]
\caption{Statistical discovery rates} % title of Table
\centering % used for centering table
\begin{tabular}{c c } % centered columns (2 columns)
\hline\hline %inserts double horizontal lines
Expected value & Name\\ [0.5ex] % inserts table heading
\hline % inserts single horizontal line
$E(V/R)$& False discovery rate (FDR)\\
$E(T/(m-R))$& False negative discovery rate (FNDR)\\
$E(T/(T+S))$& False negative rate (FNR)\\
$E(V/(U+V))$& False positive rate (FPR)\\
[1ex] % [1ex] adds vertical space
\hline %inserts single line
\end{tabular}
\label{table:discoveryrates} % is used to refer this table in the text
\end{table}
We note that the FNDR is precisely the proportion of Type II errors among the accepted null hypotheses 
, i.e. the non-discoveries.
In the datasets that we encounter within bioinformatics we often suspect
$m_1 << m$ and so if $R$, which we can observe, is relatively small, then the 
FNDR is controlled at a low level. 
As pointed out in \cite{Pawi}, apart from the FDR which measures the proportion of false positive discoveries,
we usually are interested in also controlling the FNR, 
i.e. we do not want to miss too many 
true statistical discoveries. 
We will address this by using visualization and knowledge based evaluation
to support the statistical analysis. 

In our exploration scheme presented in the next section 
we will use the step down procedure on the entire list of p-values for the
statistical test under consideration suggested by Benjamini and Hochberg in \cite{BH1}
to control the FDR.
We will also use the $q$-value, computable for each separate variable,
 introduced by Storey, see \cite{STO} and \cite{STT}. 
The $q$-value in our analyses is defined as the lowest FDR for which the particular hypothesis under consideration would be accepted under the Benjami-Hochberg step down procedure.
A practical and reasonable threshold level for the $q$-value to be informative that we will use is $q<0.2$.

\section{The basic exploration scheme}
We will look for significant signals in our data set in order to use them e.g. as a basis for 
variable selection, sample classification and clustering.

If there are enough samples one should first of all randomly partition the set of samples 
into a {\it training set} and a {\it testing set}, 
perform the analysis with the training set and then validate findings using the testing set. 
This should then ideally be repeated several times with different partitions.
With very few samples this is not
always feasible and then, in addition to 
statistical measures, one is left with  using knowledge based evaluation.
One should then remember that the 
ultimate goal of the entire exploration is to add pieces of new knowledge to an already 
existing knowledge structure. To facilitate knowledge based evaluation,
 the entire exploration scheme
is throughout guided by visualization using PCA biplots.

When looking for significant signals in the data, 
one overall rule that we follow is:
\begin{itemize}
\item{Detect and then remove the strongest present signal.}
\end{itemize}
Detect a signal can e.g. mean to find a sample cluster and 
a connected list of variables that discriminate the cluster.
We can then for example (re)classify the sample cluster in order to use the new classification to
perform more statistical tests.
After some statistical validation we then often remove a detected signal, e.g. a sample cluster, in order to avoid that 
a strong signal obscures a weaker but still detectable signal in the data.
Sometimes it is of course convenient to add the strong signal again at a later stage in order to 
use it as a reference.

We must constantly be aware of  the possibility of outliers or artifacts in our data and so we
must:
\begin{itemize}
\item{Detect and remove possible artifacts or outliers.}
\end{itemize}
An artifact is by definition 
a detectable signal that is unrelated to the basic mechanisms that we are exploring. 
An artifact can e.g. 
be created by different experimental setups, resulting in a signal in the data that 
represents different experimental conditions. 
Normally if we detect a suspected 
artifact we want to, as far as possible, eliminate the influence of the suspected artifact 
on our data. 
When we do this we must be aware that we normally 
reduce the degrees of freedom in our data.
The most common case is to eliminate a 
single nominal factor resulting in a splitting of our data in subgroups.
In this case we will 
mean-center each group discriminated by the nominal factor, 
and then analyze the data as usual, with an adjusted number of degrees of freedom.

The following basic exploration scheme is used
\begin{itemize}
\item{Reduce noise by PCA and variance filtering. 
Assess the signal/noise ratio in various low dimensional PCA projections 
and estimate the projection information contents by randomization.} 
\item{Perform statistical tests. Evaluate the statistical tests using the FDR, 
randomization and permutation tests.}
\item{Use graph-based multidimensional scaling (ISOMAP) to search for signals/clusters.}
 \end{itemize}
The above scheme is iterated until "all" significant signals are found and it is 
guided and coordinated by synchronized PCA-biplot visualizations.

\section{Some biological background concerning the example datasets.}
The rapid development of new biological measurement methods makes it possible to explore
several types of genetic alterations in a high-throughput manner.
Different types of microarrays enable researchers to simultaneously monitor the expression levels of
tens of thousands of genes. 
The available information content concerning genes, gene products and regulatory pathways 
is accordingly growing steadily.
Useful bioinformatics databases today include the Gene Ontology project (GO) \cite{AS} and 
the Kyoto Encyclopedia of Genes and Genomes (KEGG) \cite{KA} which are  initiatives with the aim of 
standardizing the representation of genes, gene products and pathways across species and databases.
A substantial collection of functionally related gene sets can also be found at the Broad Institute's 
Molecular Signatures Database (MSigDB) \cite{MSig} together with the implemented
computational method Gene Set Enrichment Analysis (GSEA) \cite{MO}, \cite{SU}. 
The method GSEA is designed to determine whether an a priori defined set of genes shows statistically
significant and concordant differences between two biological states in a given dataset.

Bioinformatic data sets are often uploaded by researchers to sites such as the National Centre for Biotechnology Information's
database Gene Expression Omnibus (GEO) \cite{GEO} or to the European Bioinformatics Institute's database ArrayExpress \cite{ARR}.
In addition data are often made available at local sites maintained by separate institutes or universities.

\section{Analysis of microarray data sets}
There are many bioinformatic and statistical challenges that remain unsolved or are only partly solved 
concerning microarray data. As explained in  \cite{Rocke}, these include normalization, variable selection, classification and clustering. 
This state of affairs is partly due to the fact that we know very little in general 
about underlying statistical distributions. 
This makes statistical robustness a key issue concerning all proposed statistical methods in this field 
and at the same time
shows that new methodologies must always be evaluated using a knowledge based approach and 
supported by accompanying new biological findings.
We will not comment on the important problems of 
normalization in what follows but refer to e.g. \cite{Autio} 
where different normalization procedures
for the Affymetrix platforms are compared.
In addition, microarray data often have a non negligible amount of missing values. 
In our example data sets we will, when needed, impute missing values using the K-nearest 
neighbors method as described in \cite{Troy}.
All visualizations and analyses 
are performed using the software Qlucore Omics Explorer \cite{Q}.

\subsection{Effects of cigarette smoke on the human epithelial cell transcriptome.}
We begin by looking at a gene expression dataset coming from the study by Spira et. al. \cite{Spira} 
of effects of cigarette smoke on the human epithelial cell transcriptome. 
It can be downloaded from 
National Center for Biotechnology Informations (NCBI)
 Gene Expression Omnibus (GEO)
(DataSet GDS534, accession no. GSE994).
It contains measurements from 75 subjects consisting
of 34 current smokers, 18 former smokers and 23 healthy never smokers. 
The platform used to collect the data was 
Affymetrix HG-U133A Chip using the Affymetrix Microarray suite to select, prepare and normalize the data,
see \cite{Spira} for details.

One of the primary goals of the investigation in \cite{Spira} was to find genes that
are responsible for distinguishing between current smokers and never smokers and also
investigate how these genes behaved when a subject quit smoking 
by looking at the expression levels 
for these genes in the group of former smokers. 
We will here focus on finding genes that discriminate the groups of 
current smokers and never smokers.

\begin{itemize}
\item {\bf We begin our exploration scheme by 
estimating the signal/noise ratio in a sample PCA projection based on 
the three first principal components.}
\end{itemize}
We use an SVD on the data correlation matrix, 
i.e. the covariance matrix for the variance normalized variables. 
In Figure \ref{fig:possibleartifactdiagnose}
we see the first three principal components for $Im \, L$ and the 75 patients plotted.
\begin{figure}[h!]
\centering
\includegraphics[width=0.5\textwidth]{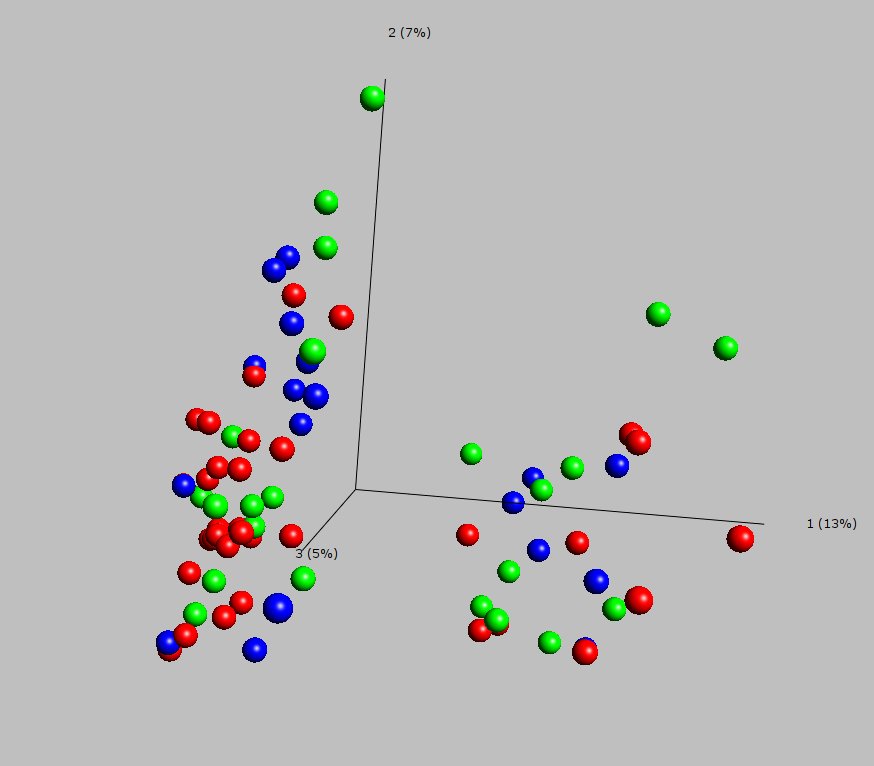}
\caption{The 34 current smokers (red), 
18 former smokers (blue) and 23 never smokers (green) projected onto the three first principal components. 
The separation into two groups is not associated with any supplied clinical annotation and is thus a suspected artifact.}\label{fig:possibleartifactdiagnose}
\end{figure}
The first three principal components contain $25\%$ 
of the total variance in the dataset and so for this $3$-D projection 
$\alpha_2(\{1,2,3\},obsr)=0.25$. 
Using randomization we estimate the expected value for a corresponding dataset 
(i.e. a dataset containing the same number of samples and variables) built on independent and normally distributed  
variables to be approximately $\alpha_2(\{1,2,3\},rand)=0.035$.
We have thus captured around $7$ times more variation  than what we would have expected if the variables
were independent and normally distributed. This indicates that we do have strong signals 
present in the dataset.
\begin{itemize}
\item {\bf Following our exploration scheme we now look for possible outliers and artifacts.}
\end{itemize}
\begin{figure}[h!]
\centering
\includegraphics[width=0.5\textwidth]{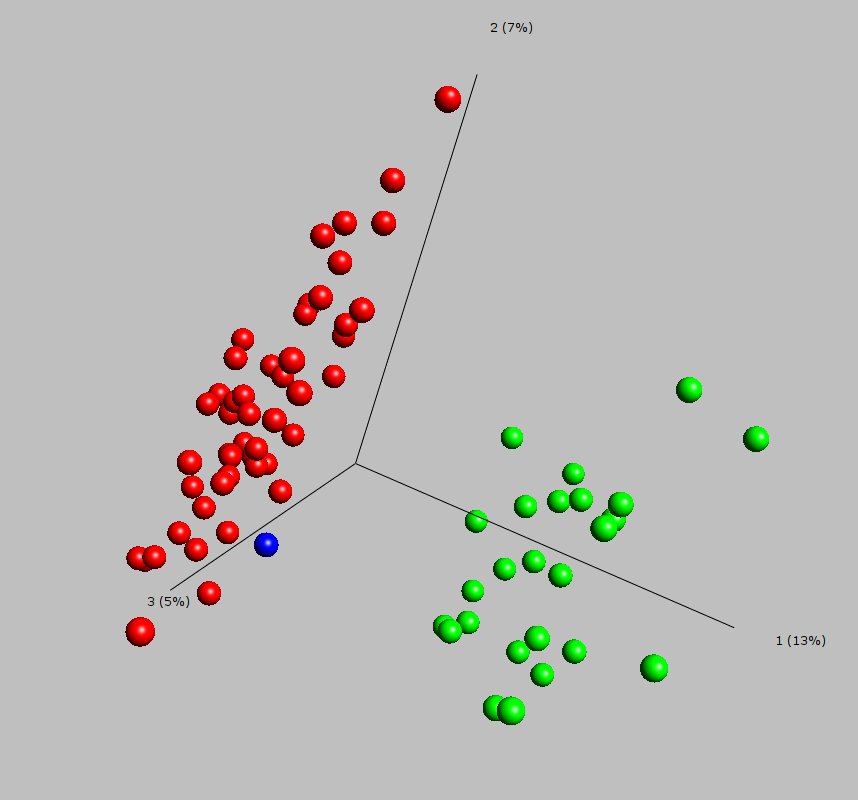}
\caption{The red samples have high description numbers ($\geq 58$) and the green samples have low  description numbers ($\leq 54$).
The blue sample has number $5$.}\label{fig:possibleartifacts}
\end{figure}
The projected subjects are colored according to smoking history, 
but it is clear from Figure \ref{fig:possibleartifactdiagnose} that 
most of the variance in the 
first principal component (containing $13\%$ of the total variance in the data)
comes from a signal that has a very weak association with smoking history.
We nevertheless see a clear splitting into two subgroups. 
Looking at supplied clinical annotations one can conclude that 
the two groups are not associated to gender, age or race traits. 
Instead one finds that all the subjects in one of the groups have low subject description numbers
whereas all the subjects except one in the other group have high subject description numbers. 
In Figure \ref{fig:possibleartifacts} we have colored the subjects according to description number.
This suspected artifact signal 
does not correspond to any in the
dataset (Dataset GDS 534, NCBIs GEO) 
supplied clinical annotation.
One can hypothesize that the description number could reflect 
for instance the order in which the samples were gathered
and thus could be an artifact.
Even if the two groups actually correspond to some interesting clinical variable, like disease state,
that we should investigate separately,
we will consider the splitting to be an artifact in our investigation.
We are interested
in using all the assembled data 
to look for genes that discriminate between current smokers and never smokers. 
We thus eliminate the suspected artifact by mean-centering the two main (artifact) groups.
After elimination of the strong artifact signal, 
the first three principal components contain "only" $17\%$ of the total variation.
\begin{itemize}
\item {\bf Following our exploration scheme we filter the genes with respect to 
variance visually searching for a possibly informative three dimensional projection.}
\end{itemize}
\begin{figure}[h!]
\centering
\includegraphics[width=0.5\textwidth]{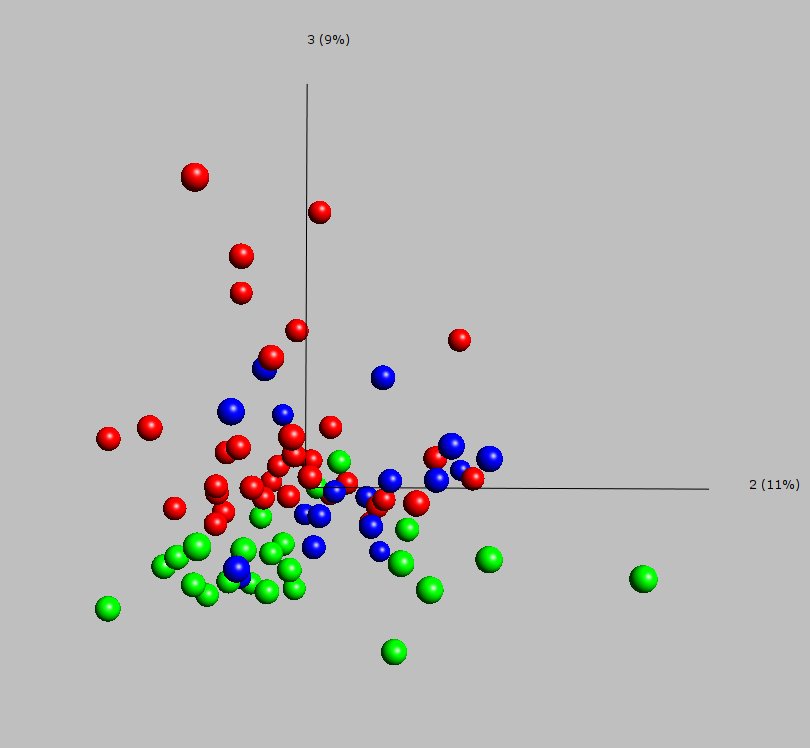}
\caption{We have filtered by variance keeping the $630$ most variable genes. 
It is interesting to see that the third principle component containing $9\%$ of the total variance separates the current smokers (red) from the never smokers (green) quite well.}\label{fig:variancefilter}
\end{figure}

\begin{table}[h!]
\caption{Top genes upregulated in the current smokers group and downregulated in the never smokers group.} % title of Table
\centering % used for centering table
\begin{tabular}{c c } % centered columns (2 columns)
\hline\hline %inserts double horizontal lines
Gene symbol& $q$-value \\ [0.5ex] % inserts table
%heading
\hline % inserts single horizontal line

NQO1	& 5.59104067771824e-08 \\
GPX2	& 2.31142232391279e-07 \\
ALDH3A1	 & 2.31142232391279e-07 \\
CLDN10	& 3.45691439169953e-06 \\
FTH1	& 4.72936617815058e-06 \\
TALDO1	& 4.72936617815058e-06 \\
TXN	& 4.72936617815058e-06 \\
MUC5AC	& 3.77806345774405e-05 \\
TSPAN1	& 4.50425200297664e-05 \\
PRDX1	& 4.58227420582093e-05 \\
MUC5AC	& 4.99131989472012e-05 \\
AKR1C2	& 5.72678146958168e-05 \\
CEACAM6	& 0.000107637125805187 \\
%AKR1C2	& 0.000156499356579749 \\
AKR1C1	& 0.000195523829628407 \\
TSPAN8	& 0.000206106293159401 \\
AKR1C3	& 0.000265342898771159 \\
%ALDOA	& 0.000289125578591603 \\
%AGR2	& 0.000310490613458269 \\
%S100P	& 0.000349955492997199  \\
%KRT8	& 0.000362720345637145 \\
%GSN	& 0.000600454852462078 \\
%CEACAM5	& 0.000600454852462078 \\
%AKR1C1	& 0.000600454852462078 \\
%CYP1B1	& 0.000664805912129266 \\
[1ex] % [1ex] adds vertical space
\hline %inserts single line
\end{tabular}
\label{table:Studentt-testup} % is used to refer this table in the text
\end{table}

\begin{table}[h!]
\caption{Top genes downregulated in the current smokers group and upregulated in the never smokers group.} % title of Table
\centering % used for centering table
\begin{tabular}{c c c} % centered columns (2 columns)
\hline\hline %inserts double horizontal lines
Gene symbol& $q$-value \\ [0.5ex] % inserts table
%heading
\hline % inserts single horizontal line
MT1G	& 4.03809377378893e-07 \\
MT1X	& 4.72936617815058e-06 \\
MUC5B	& 2.38198903402317e-05 \\
CD81	& 3.1605221864278e-05 \\
MT1L	& 3.1605221864278e-05 \\
MT1H	& 3.1605221864278e-05 \\
SCGB1A1	& 4.50425200297664e-05 \\
EPAS1	& 4.63861480935914e-05 \\
FABP6	& 0.00017793865432854 \\
MT2A	& 0.000236481909692626 \\
MT1P2	& 0.000251264650053933 \\
% C3	& 0.000251264650053933 \\
% FGFR3	& 0.000310490613458269 \\
% MT1E	& 0.000626133311285203 \\
[1ex] % [1ex] adds vertical space
\hline %inserts single line
\end{tabular}
\label{table:Studentt-testdown} % is used to refer this table in the text
\end{table}

When we filter down to the 
$630$ most variable genes,  
the three first principal components have an $L^2$-projection content
of $\alpha_2(\{1,2,3\}) =0.42$, 
whereas by estimation using randomization we would have expected it to be $0.065$. 
The projection in Figure \ref{fig:variancefilter} 
is thus probably informative.
We have again colored the samples according to smoking history as above.
The third principal component, containing $9\%$ of the total variance,
 can be seen to quite decently separate
the current smokers from the never smokers. We note that this was impossible to achieve 
without removing the artifact signal
since the artifact  signal completely obscured this separation.

\begin{figure}[h!]
\centering
\includegraphics[width=1.0\textwidth]{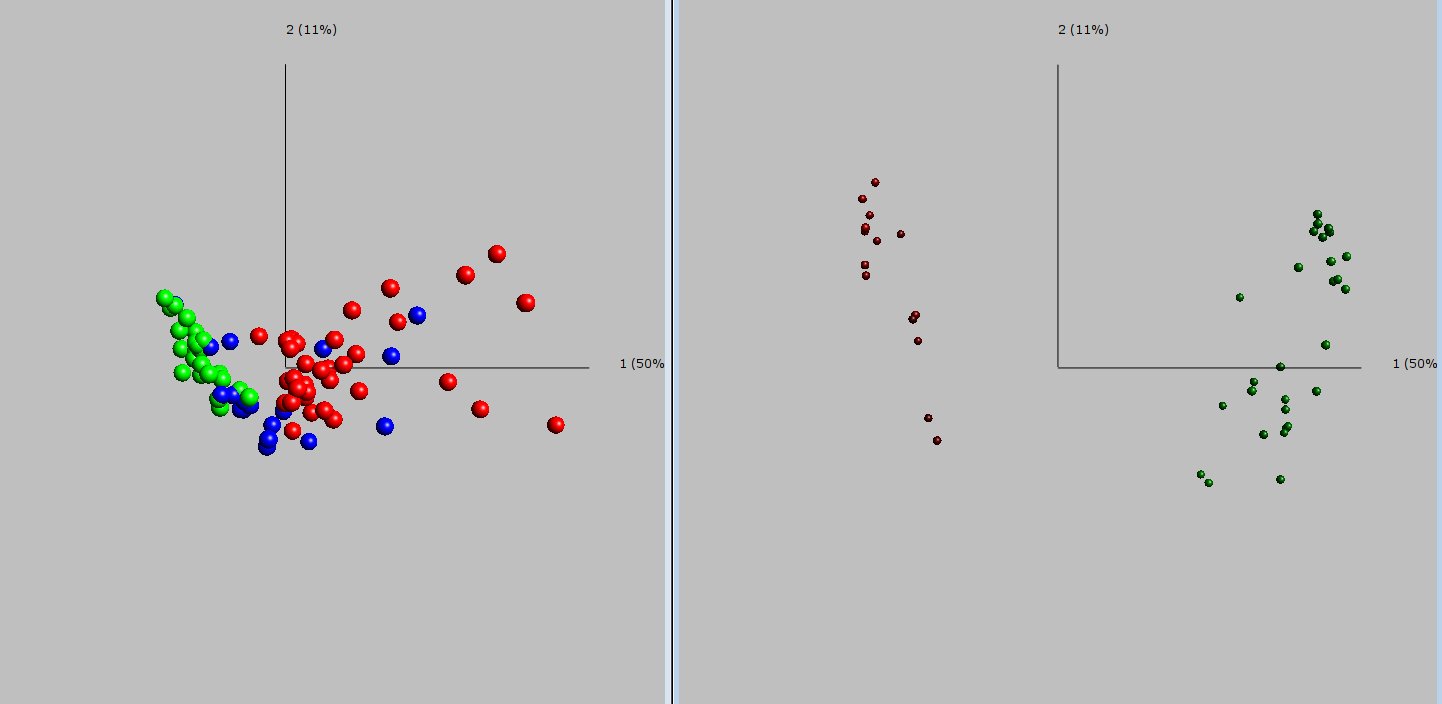}
\caption{A synchronized biplot showing samples to the left 
and variables to the right.}\label{fig:filterqbiplot}
\end{figure}

\begin{itemize}
\item {\bf Using the variance filtered list of 630 genes as a basis, 
following our exploration scheme, 
we now perform a series of Student t-tests between the groups of current smokers and never smokers, i.e. $34+23=57$
different subjects.}
\end{itemize}
For a specific level of significance 
we compute the $3$-dimensional 
(i.e. we let $S=\{1,2,3\}$)
$L^2$-projection  content resulting when we keep all the
rejected null hypotheses, i.e. statistical discoveries. 
For a sequence of t-tests parameterized by the  level of significance we now try to 
find a small level of significance and at the same time 
an observed $L^2$-projection content with a large quotient compared to 
the expected projection content estimated by randomization.
We supervise this procedure visually using three dimensional PCA-projections
looking for visually clear patterns.
For a level of significance of $0.00005$, 
leaving a total of $43$ genes (rejected nulls) and an FDR of $0.0007$ we have 
$\alpha_2(\{1,2,3\}, obsr)=0.71$ 
whereas the expected projection content for randomized data $\alpha_2(\{1,2,3\}, rand)=0.21$.
We have thus captured more than $3$ times of the expected projection content and at the same time
approximately $0.0007 \times 43=0.0301$ genes are false discoveries
and so with high probability we have found $43$ potentially important biomarkers.
We now visualize  all $75$ subjects using these $43$ genes as variables. 
In Figure \ref{fig:filterqbiplot} we see a 
synchronized biplot with samples to the left and variables to the right. 
The sample plot shows a perfect separation of current smokers and never smokers.
In the variable plot we see genes (green) that are 
upregulated in the current smokers group to the far right.
The top genes according to $q$-value for the Student t-test between current smokers and never smokers, 
that are upregulated in the current smokers group and downregulated in the never smokers group, are given in
Table \ref{table:Studentt-testup}. 
In Table \ref{table:Studentt-testdown} we list the top genes that are downregulated in the current smokers group and 
upregulated among the never smokers.

\subsection{Analysis of various muscle diseases.}
In the study by Bakay et. al. \cite{Bakay} the authors 
studied 125 human muscle biopsies from 13 diagnostic groups 
suffering from various muscle diseases. 
The platforms used were Affymetrix  U133A and U133B chips.
The dataset can be downloaded from NCBIs Gene Expression Omnibus 
(DataSet GDS2855, accession no. GSE3307). 
We will analyze the dataset looking for phenotypic classifications 
and also looking for biomarkers for the different phenotypes.
\begin{itemize}
\item{ \bf We first use $3$-dimensional PCA-projections of the samples of the data correlation matrix, filtering the genes 
with respect to variance and 
visually searching for clear patterns.}
\end{itemize}

\begin{figure}
  \centering
  \subfloat[A three dimensional PCA plot capturing $46\%$ of the total variance.]{\label{fig:musclevarfilter300plot}\includegraphics[width=0.5\textwidth]{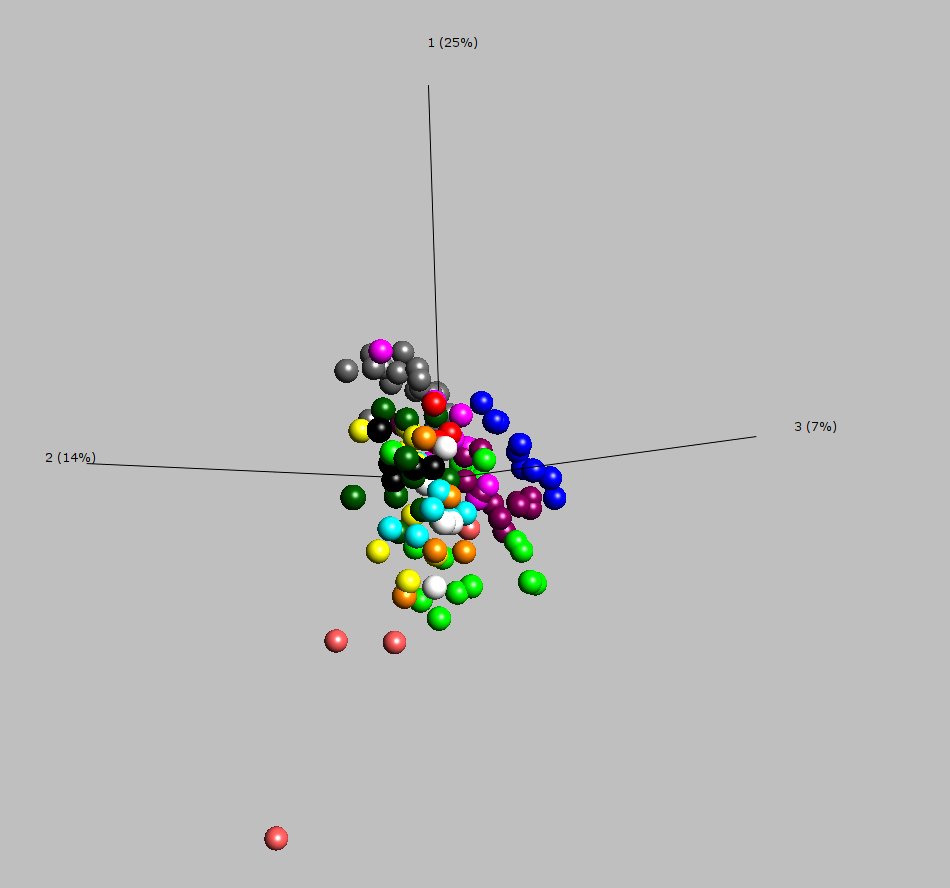}}                
  \subfloat[Color legend]{\label{fig:colorlegend}\includegraphics[width=0.4\textwidth]{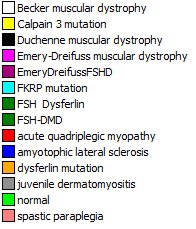}}
  \caption{PCA-projection of samples 
based on the variance filtered top 300 genes. 
Three out of four subjects in the group Spastic paraplegia (Spg)
clearly distinguish themselves. These three suffer from Spg-4, while the remaining Spg-patient
suffers from Spg-7.}
  \label{fig:musclevarfilter300}
\end{figure}

When filtering out the 
$300$ genes having most variability
over the sample set we see several samples clearly distinguishing themselves and we 
capture $46\%$ of the total variance compared to the, by randomization estimated, 
expected
$6\%$. The plot in 
Figure \ref{fig:musclevarfilter300} thus contains strong signals. 
Comparing with the color legend
we conclude that the patients suffering from
spastic paraplegia (Spg) contribute a strong signal. 
More precisely, three of the subjects suffering from the 
variant Spg-4 clearly distinguish themselves, while the remaining patient in the 
Spg-group suffering from Spg-7 falls close to the rest of the samples.

\begin{table}[h!]
\caption{Top genes upregulated in the Spastic paraplegia group (Spg-4).} % title of Table
\centering % used for centering table
\begin{tabular}{c c c} % centered columns (2 columns)
\hline\hline %inserts double horizontal lines
Gene symbol& $q$-value \\ [0.5ex] % inserts table
%heading
\hline % inserts single horizontal line
RAB40C	& 0.0000496417 \\
SFXN5	& 0.000766873 \\
CLPTM1L	& 0.00144164\\
FEM1A	& 0.0018485 \\
HDGF2	& 0.00188435 \\
WDR24	& 0.00188435\\
NAPSB	& 0.00188435\\
ANKRD23	& 0.00188435\\
[1ex] % [1ex] adds vertical space
\hline %inserts single line
\end{tabular}
\label{table:Studentt-testSPP} % is used to refer this table in the text
\end{table}
\begin{itemize}
\item {\bf We perform Student t-tests between the spastic paraplegia group and the normal group.}
\end{itemize}
As before we now, in three dimensional PCA-projections, visually search for clearly distinguishable patterns
in a sequence of Student t-tests parametrized by level of significance, while at the same time trying to obtain a small FDR.
At a level of significance of $0.00001$,
leaving a total of $37$ genes (rejected nulls) with an FDR of $0.006$, the first three principal components capture $81\%$ of the variance compared to the, by randomization, expected $31\%$. 
Table \ref{table:Studentt-testSPP} 
lists the top genes upregulated in the group
spastic paraplegia. 
We can add that these genes are all strongly upregulated for the three
particular subjects suffering from Spg-4, while that pattern is less clear
for the patient suffering from Spg-7.

\begin{itemize}
\item {\bf In order to find possibly obscured signals, we now remove the Spastic paraplegia group 
from the analysis.}
\end{itemize}
We also remove the Normal group from the analysis since we really 
want to compare the different diseases. 
Starting anew with the entire set of genes, 
filtering with respect to variance, we visually obtain clear patterns for the $442$ most variable genes.
The first three principal components capture $46\%$ of the total variance compared to the, by randomization estimated,
 expected $6\%$.
\begin{itemize}
\item {\bf Using these $442$ most variable genes as a basis for the analysis, 
we now construct a graph connecting every sample with its
two nearest (using euclidean distances in the $442$-dimensional space) neighbors.}
\end{itemize}
As described in the section on multidimensional scaling above, we now compute
 geodesic distances in the graph between samples, 
 and construct a resulting distance (between samples) matrix. 
We then convert this distance matrix to 
a corresponding covariance matrix and finally 
perform a PCA on this covariance matrix. 
The resulting plot (together with the used graph)
of the so constructed three dimensional PCA-projection is
depicted in Figure \ref{fig:muscleisomap}. 
\begin{figure}[h!]
\centering
\includegraphics[width=0.5\textwidth]{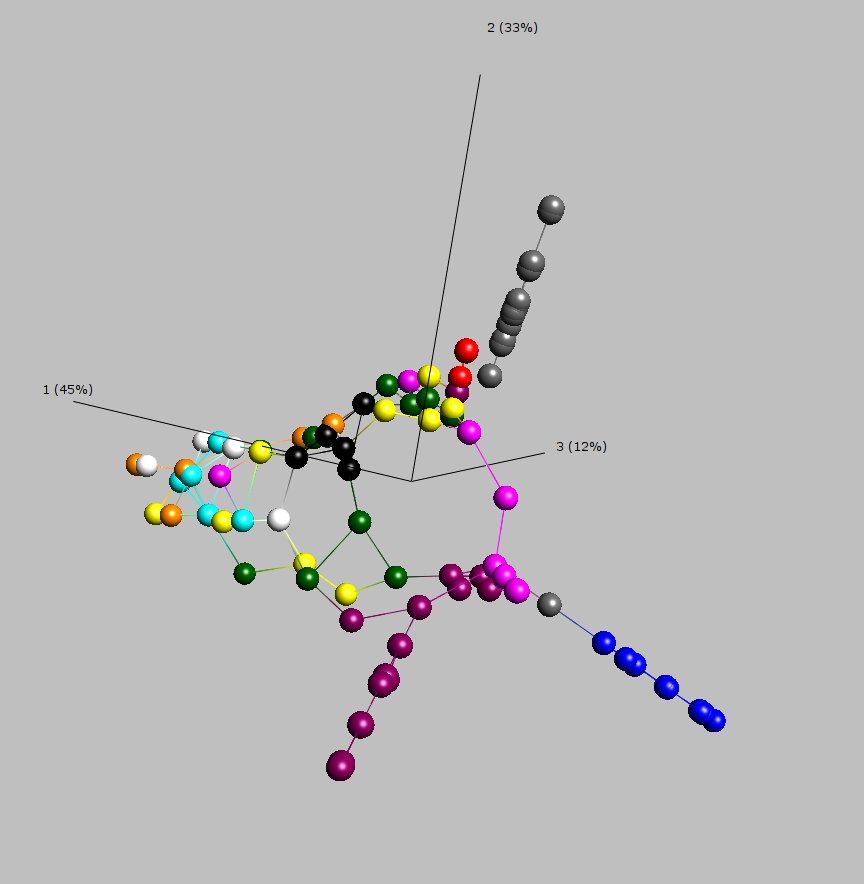}
\caption{Effect of the ISOMAP-algorithm. 
We can identify a couple of clusters corresponding to 
the groups juvenile dermatomyositis, amyotophic lateral sclerosis, 
acute quadriplegic myopathy and Emery Dreifuss FSHD.}\label{fig:muscleisomap}
\end{figure}

Comparing with the Color legend in Figure \ref{fig:musclevarfilter300},  
we clearly see that the groups juvenile dermatomyositis, amyotophic lateral sclerosis, 
acute quadriplegic myopathy and also Emery-Dreifuss FSHD distinguish themselves. 

One should now go on using Student t-tests to find 
biomarkers (i.e. genes) distinguishing these different groups of patients, 
then eliminate these distinct groups and go on searching for more
structure in the dataset.

\subsection{Pediatric Acute Lymphoblastic Leukemia (ALL)} 
We will finally analyze a dataset 
consisting of gene expression profiles from 132 different patients, all suffering from some type of pediatric acute lymphoblastic leukemia (ALL). 
For each patient the expression levels of 22282 genes are analyzed. 
The dataset comes from the study by Ross et. al. \cite{R} and the primary data
are available at the St. Jude Children's Research Hospital's website \cite{StJ}.
The platform used to collect this example data set was Affymetrix HG-U133 chip,
using the Affymetrix Microarray suite to select, prepare and normalize the data. 

As before we start by performing an SVD on the data correlation matrix visually searching for interesting patterns and assessing the signal to noise ratio
by comparing the actual $L^2$-projection content in the real world data projection
with the expected $L^2$-projection content in corresponding randomized data.
\begin{itemize}
\item {\bf We filter the genes with respect to variance, 
looking for strong signals.}
\end{itemize}

\begin{figure}[h!]
\centering
\subfloat[Projection capturing $38\%$ of the total variance.]{\label{fig:filter873}\includegraphics[width=0.5\textwidth]{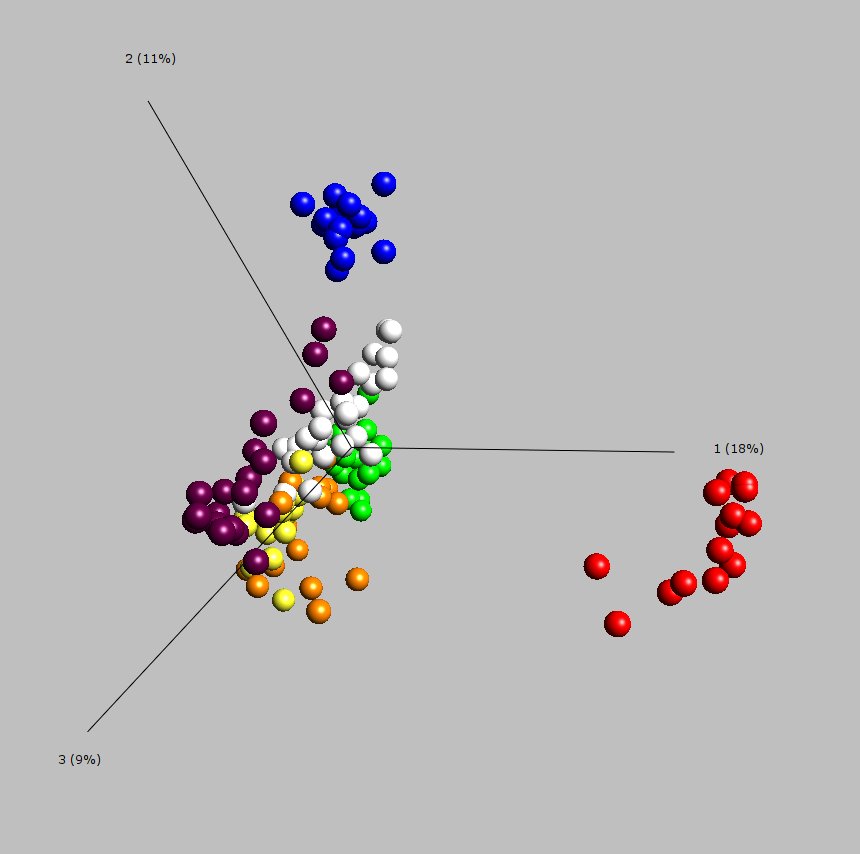}}               
\subfloat[Color legend]{\label{fig:colorlegendross}\includegraphics[width=0.25\textwidth]{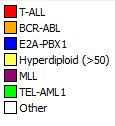}}
\caption{Variance filtered PCA-projection of the correlation datamatrix based on $873$ genes.
The group T-ALL clearly distinguish itself.}
\label{fig:rossvarfilter873}
\end{figure}

In Figure \ref{fig:rossvarfilter873}
we see a plot of a three dimensional projection using the $873$ most variable genes as a basis 
for the analysis. 
We clearly see that the group T-ALL 
is mainly responsible for the signal resulting in the first principal component occupying $18 \%$ of the
total variance. 
In fact by looking at supplied annotations we can conclude that 
all of the other subjects in the dataset are suffering from B-ALL, the other main ALL type.

\begin{itemize}
\item {\bf We now perform Student t-tests between the group T-ALL and the rest. 
We parametrize by level of significance and 
visually search for clear patterns.}
\end{itemize}

\begin{figure}[h!]
\centering
\includegraphics[width=1.0\textwidth]{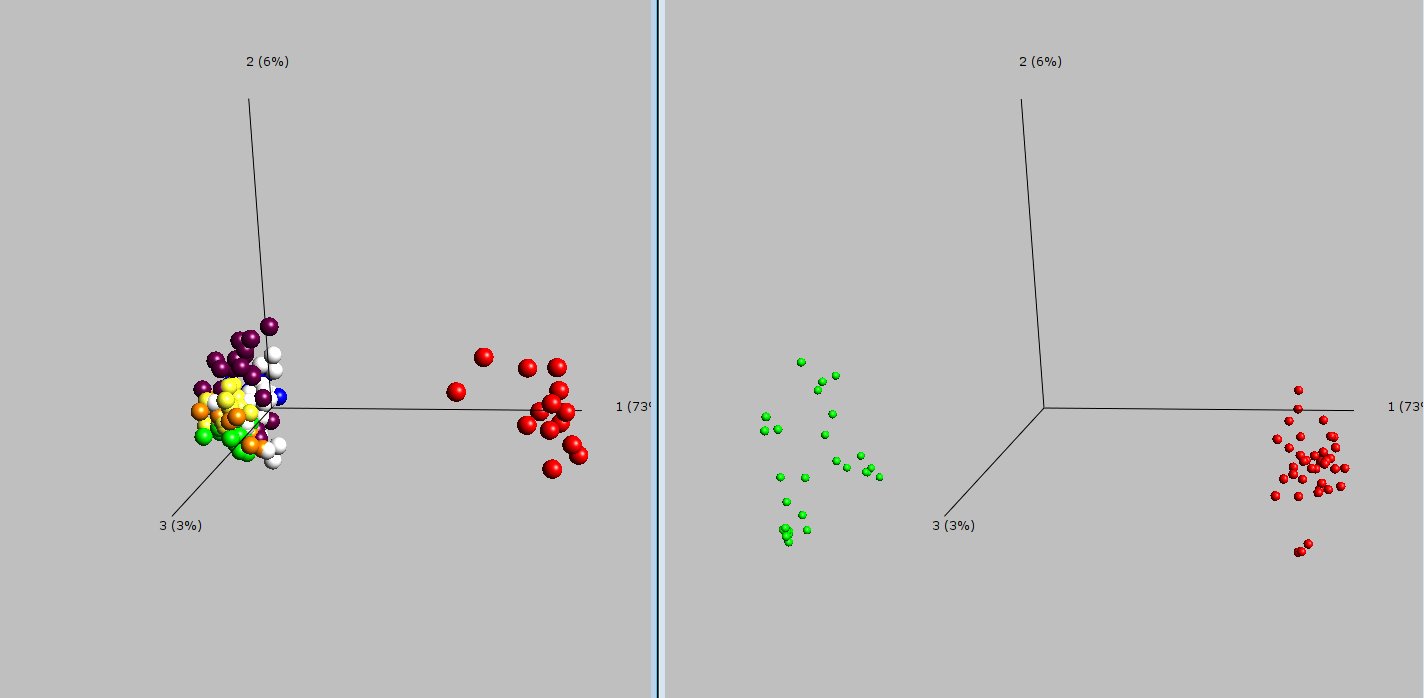}
\caption{FDR $=1.13e-24$. 
A synchronized biplot showing samples to the left 
and genes to the right. 
The genes are colored according to their expression level in the T-ALL group. 
Red $=$ upregulated and green $=$ downregulated.}\label{fig:biplotT-ALLvsB-ALL}
\end{figure}

In Figure \ref{fig:biplotT-ALLvsB-ALL} we see a biplot 
based on the 70 genes that best discriminate between
T-ALL and the rest. 
The FDR is extremely low $FDR=1.13e-24$ telling us that with a very high
probability the genes found are relevant discoveries. 
The most significantly upregulated genes in the T-ALL group are 

CD3D, CD7, TRD@, CD3E, SH2D1A and TRA@.

\noindent
The most significantly downregulated genes in the T-ALL group are 

CD74, HLA-DRA, HLA-DRB, HLA-DQB and BLNK.

By comparing with gene-lists from the MSig Data Base (see \cite{MSig}) 
we can see that the genes that are upregulated
in the T-ALL group (CD3D, CD7 and CD3E) are represented in lists of genes connected 
to lymphocyte activation and lymphocyte differentiation.

\begin{figure}[h!]
\centering
\includegraphics[width=0.5\textwidth]{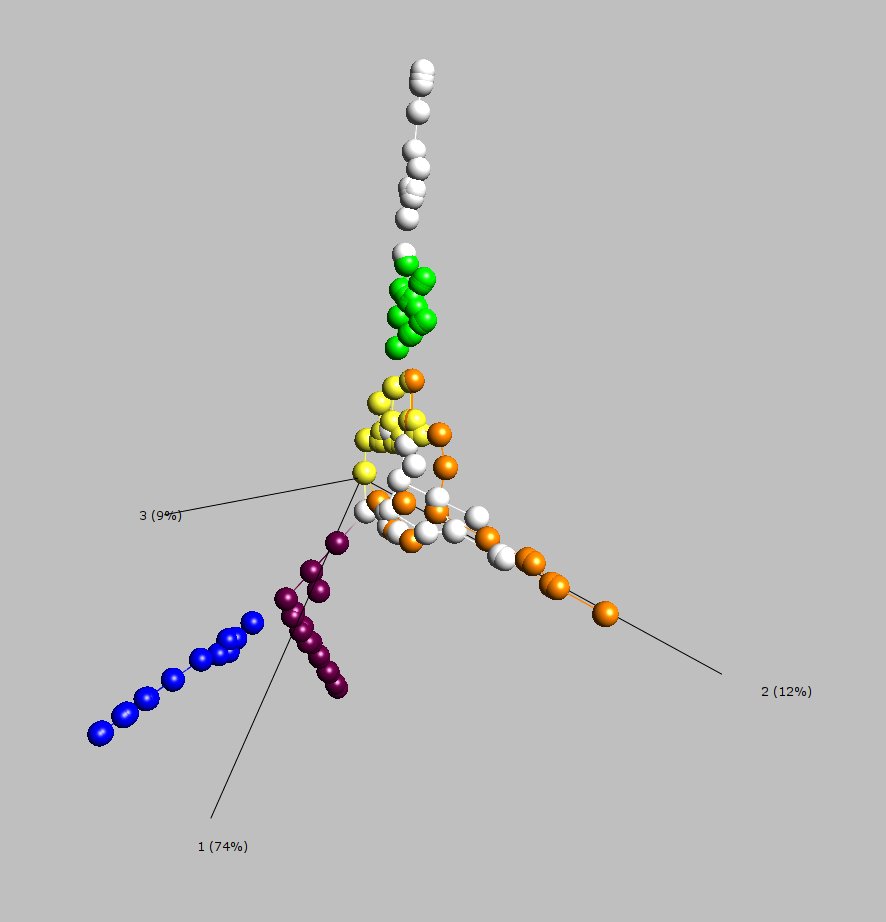}
\caption{Effect of the ISOMAP-algorithm. We can clearly see the groups E2A-PBX1, MLL and TEL-AML1. The group TEL-AML1
is connected to a subgroup of the group called Other (white).}\label{fig:ALLisomap}
\end{figure}

\begin{itemize}
\item {\bf We now remove the group T-ALL from the analysis and search 
for visually clear patterns among three dimensional PCA-projections 
filtrating the genes with respect to variance.} 
\end{itemize}
Starting anew with the entire list of genes, filtering with respect to variance,
a clear pattern is obtained for the $226$ most variable genes. We capture $43\%$ of the total variance
as compared to the expected $6.5\%$. We thus have strong signals present.

\begin{itemize}
\item {\bf Using these $226$ most variable genes as a basis for the analysis, 
we now construct a graph connecting every sample with its
two nearest neighbors.}
\end{itemize}
We now perform the ISOMAP-algorithm with respect to this graph.
The resulting plot (together with the used graph)
of the so constructed three dimensional PCA-projection is
depicted in Figure \ref{fig:ALLisomap}. 
We can clearly distinguish the groups E2A-PBX1, MLL and TEL-AML1. The group TEL-AML1
is connected to a subgroup of the group called Other. 
This subgroup actually corresponds to the 
Novel Group discovered in the study by Ross et.al. \cite{R}.
Note that by using ISOMAP we discovered this Novel subgroup only by variance filtering the genes
showing that ISOMAP is a useful tool for visually supervised clustering.

\vskip 20 mm

{\bf Acknowledgement}
I dedicate this review article to Professor Gunnar Sparr.
Gunnar has been a role model for me, and many other 
young mathematicians, 
of a pure mathematician that evolved into contributing serious
applied work. 
Gunnar's help, support and general encouragement 
have been very important during my own development within the field of mathematical modeling.
Gunnar has also been one of the ECMI pioneers introducing Lund University
to the ECMI network.

I sincerely thank Johan R{\aa}de for helping me 
to learn almost everything I know about data exploration.
Without him the here presented work would truly 
not have been possible.
Applied work is best done in collaboration and I am blessed with
Thoas Fioretos as my long term collaborator 
within the field of molecular biology.
I am grateful for what he has tried to teach me
and I hope he is willing to continue to try.
Finally I thank Charlotte Soneson for reading this work 
and, as always, giving very valuable feed-back.

\vskip 20 mm

{\bf Keywords}: multivariate statistical analysis, principal component analysis, biplots, multidimensional scaling, multiple hypothesis testing, false discovery rate, microarray, bioinformatics\\


\begin{thebibliography}{WWWWWW99}
\bibitem{OA} O. Alter, P. Brown, D. Botstein {\it Singular value decomposition for genome-wide expression data processing and modeling.} Proceedings of the National Academy of Science 97 (18), (2000) pp. 10101--10106.
\bibitem{AN1} T.W. Anderson {\it Asymptotic theory for principal component analysis} Ann. Math. Statist. 34, (1963) pp 122--148.
\bibitem{AN} T. W. Anderson {\it An introduction to multivariate statistical analysis.} 3rd ed. (2003), Wiley, Hoboken, NJ.
\bibitem{ARR} The European Bioinformatics Institute's database ArrayExpress: http://www.ebi.ac.uk/microarray-as/ae/
\bibitem{AS} M. Ashburner et.al. {\it The Gene Ontolgy Consortium. Gene Ontology: tool for the unification of biology.} Nat. Genet. 25, (2000), pp. 25--29.
\bibitem{Autio} R. Autio et. al. {\it Comparison of Affymetrix data normalization methods using 6,926 experiments across five array generations.} 
BMC Bioinformatics  Vol. 10, suppl.1 S24 (2009).
\bibitem{BAI} Z. D. Bai {\it Methodologies in spectral analysis of large dimensional random matrices, A review.} Stat. Sinica 9 (1999) pp. 611--677.
\bibitem{B1} E. Bair, R. Tibshirani {\it Semi-supervised methods to predict patient survival from gene expression data.} PLOS Biology 2, (2004), pp. 511--522.
\bibitem{B2} E. Bair, T. Hastie, D. Paul, R. Tibshirani {\it Prediction by supervised principle components.} J. Amer. Stat. Assoc. 101, (2006), pp. 119--137.

\bibitem{Bakay} M. Bakay et al. {\it Nuclear envelope dystrophies show a transcriptional fingerprint suggesting disruption of Rb-MyoD pathways in muscle regeneration.} Brain (2006), 129(Pt 4), pp. 996-1013. 
\bibitem{Barry} W.T. Barry, A.B. Nobel, F.A. Wright {\it A statistical framework for testing functional categories in microarray data.} The Annals of Appl. Stat. Vol. 2, No 1, (2008), pp. 286--315.
\bibitem{BH1} Y. Benjamini, Y. Hochberg {\it Controlling the false discovery rate: 
a practical and powerful approach to multiple testing} J. R. Stat. Soc. Ser. B, 57, (1995) pp. 289--300.
\bibitem{BH2} Y. Benjamini, Y. Hochberg {\it On the adaptive control of the false discovery rate in multiple testing with independent statistics.} J. Edu. Behav. Stat., 25, (2000), pp. 60--83.
\bibitem{BEY} Y. Benjamini, D. Yekutieli {\it The control of the false discovery
rate in multiple testing under dependency.} The Annals of Statistics, 29,(2001), pp. 1165--1188.
\bibitem{Braak} C. J.F. Ter Braak 
{\it Interpreting canonical correlation analysis through biplots of structure correlations and weights.}
Psychometrika Vol. 55, No 3 (1990) pp. 519--531.
\bibitem{Chen} X. Chen, L. Wang, J.D. Smith, B. Zhang 
{\it Supervised principle component analysis for gene set enrichment of microarray data with continuous or survival outcome.}
Bioinformatics Vol 24, no 21, (2008) pp. 2474--2481.
\bibitem{B3} P. Debashis, E. Bair, T. Hastie, R. Tibshirani {\it "Preconditioning" for feature selection and regression in high-dimensional problems.} The Annals of Statistics, Vol. 36, No 4, (2008), pp. 1595--1618.
\bibitem{DIAC} P. Diaconis {\it Patterns in eigenvalues: The 70´th Josiah Willard Gibbs Lecture.} Bulletin of the AMS Vol. 40, No 2 (2003) pp. 155--178.
\bibitem{GEO} National Centre for Biotechnology Information's
database Gene Expression Omnibus (GEO): http://www.ncbi.nlm.nih.gov/geo/

\bibitem{GabKR1} K. R. Gabriel {\it The biplot graphic display of matrices with application to principal component analysis.}
Biometrika 58, (1971) pp. 453--467.
\bibitem{GabKR2} K.R. Gabriel {\it Biplot.} In S. Kotz; N.L. Johnson (Eds.) {\it Encyclopedia of Statistical Sciences} Vol 1 (1982)
New York; Wiley pp. 263--271.
\bibitem{GoHa} J.C. Gower, D.J. Hand {\it Biplots.} Monographs on Statistics and Applied probability 54;  Chapman \& Hall; London (1996).

\bibitem{H} H. Hotelling {\it The generalization of Student's ratio}, Ann. Math. Statist. (1931), Vol. 2, pp 360–-378.
\bibitem{H2} H. Hotelling {\it Analysis of a complex of statistical variables into principal components.} J. Educ. Psychol. 24, (1933), pp 417--441; pp 498--520.
\bibitem{P} K. Pearson {\it On lines and planes of closest fit to systems of points in space.} Phil. Mag. (6), 2, (1901), pp 559--572.
\bibitem{IMJ1} I. M. Johnstone {\it On the distribution of the largest eigenvalue in Principle Components Analysis.}
Ann. of Statistics Vol 29, No 2 (2001) pp. 295--327.
\bibitem{IMJ2} I. M. Johnston {\it High dimensional statistical inference and random matrices.} Prooc. of the Intern. congress of Math, Madrid, Spain 2006, (EMS 2007).
\bibitem{KA} M. Kanehisa, S. Goto {\it KEGG:Kyoto Encyclopedia of Genes and Genomes.} Nucleic Acids Res. 28, (2000),  pp. 27--30.
\bibitem{Kar} K. Karhunen, {\it {\"U}ber lineare Methoden in der Wahrscheinlichkeitsrechnung}, Ann. Acad. Sci. Fennicae. Ser. A. I. Math.-Phys., No 37, (1947), pp. 1–79.
\bibitem{KARO} N. El Karoui {\it Spectrum estimation for large dimensional covariance matrices using random matrix theory.} The Ann. of Statistics Vol 36, No 6 (2008) pp. 2757--2790.
\bibitem{Khatri} P. Khatri, S. Draghici, {\it Ontological analysis of gene expression data: current tools, limitations, and open problems.} 
Bioinformatics, Vol. 21, no 18, (2005), pp. 3587--3595.
\bibitem{KIM} B.S. Kim et. al. {\it Statistical methods of translating microarray data into clinically relevant diagnostic information in colorectal cancer.} Bioinformatics, 21, (2005), pp. 517--528.
\bibitem{SWK} S.W. Kong, T.W. Pu, P.J. Park {\it A multivariate approach for integrating genome-wide expression data and biological knowledge.}
Bioinformatics Vol. 22, no 19, (2006), pp. 2373--2380.
\bibitem{Loeve} M. Lo{\`e}ve, {\it Probability theory.} Vol. II, 4th ed., Graduate Texts in Mathematics, Vol. 46, Springer-Verlag, 1978, ISBN 0-387-90262-7.
\bibitem{MIRS} L. Mirsky {\it Symmetric gauge functions and unitarily invariant norms.} The quarterly journal of mathematics Vol 11 No 1 (1960) pp. 50--59.
\bibitem{MSig} The Broad Institute's 
Molecular Signatures Database (MSigDB): http://www.broadinstitute.org/gsea/msigdb/
\bibitem{MO} V. K. Mootha et. al. {\it Pgc-1 alpha-responsive genes involved in oxidative phosphorylation are coordinately downregulated in human diabetes}
Nat. Genet. 34, (2003), pp 267--273.
\bibitem{Nils} J. Nilsson, T. Fioretos, M. H\"oglund, M. Fontes {\it Approximate geodesic distances reveal biologically relevant structures in microarray data.}
Bioinformatics Vol. 20, no 6, (2004), pp 874--880.
\bibitem{Pawi} Y. Pawitan, S. Michiels, S. Koscielny, A. Gusnanto and
A. Ploner {\it False discovery rate, sensitivity and sample size for microarray
studies} Bioinformatics Vol. 21 no. 13 (2005), pp.  3017–-3024.
\bibitem{RAO} C. R. Rao {\it Separation theorems for singular values of matrices and their applications in multivariate analysis.} J. of Multivariate analysis 9 (1979) pp. 362--377.
\bibitem{RTG} D. Rasch, F. Teuscher, V. Guiard {\it How robust are tests for two independent samples?} Journal of statistical planning and inference 137 (2007) pp. 2706--2720.
\bibitem{Rivals} I. Rivals, L. Personnaz, L. Taing, M-C Potier {\it Enrichment or depletion of a GO category within a class of genes: which test?}
Bioinformatics Vol. 23, no 4, (2007), pp. 401--407.
\bibitem{Rocke} D.M. Rocke, T. Ideker, O. Troyanskaya, J. Queckenbush, J. Dopazo 
{\it Editorial note: Papers on normalization, variable selection, classification or clustering of microarray data} Bioinformatics Vol. 25, no 6, (2009), pp. 701--702.
\bibitem{R} M.E. Ross et. al. {\it Classification of pediatric acute lymphoblastic leukemia
by gene expression profiling} Blood 15 October 2003, Vol 102, No 8, pp 2951--2959.
\bibitem{Q} Qlucore Omics Explorer, Qlucore AB, www.qlucore.com
\bibitem{Spira} A. Spira et al. {\it Effects of Cigarette Smoke on the Human Airway Epithelial Cell Transcriptome}
Proceedings of the National Academy of Sciences,
Vol. 101, No. 27 (Jul. 6, 2004), pp. 10143-10148.
\bibitem{Stew} G.W. Stewart {\it On the early history of the singular value decomposition} SIAM Review, Vol 35, no 4 (1993) pp. 551--566.
\bibitem{STO} J.D. Storey {\it A direct approach to false discovery rates.} J.R. Stat. Soc. Ser. B, 64, (2002), pp. 479--498.
\bibitem{STT} J. D. Storey, R. Tibshirani {\it Statistical significance for genomewide studies.} Proc. Natl. Acad. Sci. USA, 100, (2003), pp. 9440--9445.
\bibitem{StJ} St. Jude Children's Research Hospital: http://www.stjuderesearch.org/data/ALL3/index.html
\bibitem{SU} A. Subramanian et.al. {\it Gene set enrichment analysis: a knowledgebased approach for interpreting genome wide expression profiles.} 
Proc. Natl. Acad. Sci. USA, 102, (2005) pp. 15545--15550.
\bibitem{Tenen} J. B. Tenenbaum, V. de Silva and J. C. Langford {\it A Global Geometric Framework for Nonlinear
Dimensionality Reduction.} Science 290 (22 Dec. 2000): pp. 2319-2323.

\bibitem{Troy} O. Troyanskaya et. al. {\it Missing value estimatin methods for DNA microarrays.} Bioinformatics Vol. 17, no 6, (2001) pp. 520--525.
\bibitem{Yin} Y. Yin, C. E. Soteros, M.G. Bickis {\it A clarifying comparison of methods for controlling the false discovery rate.} Journal of statistical planning and inference 139 (2009) pp. 2126--2137.
\bibitem{YBK} Y. Q. Yin, Z. D. Bai, P.R. Krishnaiah {\it On the limit of the largest eigenvalue of the large dimensional sample covariance matrix} Probab. Theory Related Fields 78 (1988) pp. 509--521.
\end{thebibliography}
\end{document}